\documentclass[11pt]{amsart}

\title[Local isometric embeddings of Lie groups]
{On local isometric embeddings of three-dimensional Lie groups
}

\author[Y.\ Agaoka]{Yoshio Agaoka}
\author[T.\ Hashinaga]{Takahiro Hashinaga}
\address[Y.\ Agaoka]{Department of Mathematics, Hiroshima University, 
Higashi-Hiroshima 739-8521, Japan}
\address[T.\ Hashinaga]{National Institute of Technology, Kitakyushu College, 
Kitakyushu, Fukuoka 802-0985, Japan}

\email[Y.\ Agaoka]{agaoka@hiroshima-u.ac.jp}
\email[T.\ Hashinaga]{hashinaga@kct.ac.jp}


\usepackage{a4}
\usepackage{amsmath,amssymb,amscd,amsthm,cases}
\usepackage{amsfonts}
\usepackage[dvips]{graphicx}

\newtheorem{thm}{Theorem}[section]
\newtheorem{prop}[thm]{Proposition}
\newtheorem{lem}[thm]{Lemma}

\theoremstyle{definition}
\newtheorem{defi}[thm]{Definition}
\newtheorem{rem}[thm]{Remark}
\newcommand{\g}{\mathfrak{g}}

\newcommand{\M}{\widetilde{\mathfrak{M}}}
\newcommand{\PM}{\mathfrak{PM}}
\newcommand{\A}{\mathbb{R}^{\times}\mathrm{Aut}(\g)}
\newcommand{\Aut}{\mathrm{Aut}(\g)}

\newcommand{\naiseki}{\langle~ ,~ \rangle}
\newcommand{\R}{\mathbb{R}}

\newcommand{\GL}{\mathrm{GL}_n(\R)} 
\newcommand{\OO}{\mathrm{O}(n)} 

\numberwithin{equation}{section}

\begin{document}

\begin{abstract}
Due to Janet-Cartan's theorem, any analytic Riemannian manifolds can be locally 
isometrically embedded into a sufficiently high dimensional Euclidean space.
However, for an individual Riemannian manifold $(M,g)$, it is in general hard to determine 
the least dimensional Euclidean space into which $(M,g)$ can be locally isometrically embedded, 
even in the case where $(M,g)$ is homogeneous. 
In this paper, when the space $(M,g)$ is locally isometric to a three-dimensional Lie group 
equipped with a left-invariant Riemannian metric, we classify all such spaces that can be locally 
isometrically embedded into the four-dimensional Euclidean space. 
Two types of algebraic equations, the Gauss equation and the derived Gauss 
equation, play an essential role in this classification.
\end{abstract}

\maketitle

\section{Introduction}

In this paper we study the local isometric embedding problem of three-dimensional Lie groups $G$ 
equipped with left invariant Riemannian metrics.
We classify all left invariant Riemannian metrics on $G$ that can be locally isometrically 
embedded into the four-dimensional Euclidean space $\mathbb{R}^4$.
The main results are summarized in Theorem~\ref{Main} below.

We first review some known results on local isometric embedding problems.
Let $(M,g)$ be an $n$-dimensional Riemannian manifold. 
An embedding $F$ of $(M,g)$ into the Euclidean space $\mathbb{R}^m$ is called an \textit{isometric embedding} if 
the Riemannian metric $g$ is equal to the pullback of the standard Euclidean metric by $F$.  
Namely the differential equation
\begin{align*}
\langle dF, dF \rangle_{\mathbb{R}^m} =g
\end{align*} 
holds. 
By virtue of the famous Jannet-Cartan's theorems (\cite{Car, J2}), it is known that any $n$-dimensional analytic 
Riemannian manifold can be locally isometrically embedded into the Euclidean space of 
dimension $\frac{n(n+1)}{2}$.
(Later, as is also well known, Nash established the global isometric embedding theorem (\cite{N}) in 
the $C^{\infty}$-category.)

On the other hand, in contrast to the above general theory, not so many results are known for 
individual Riemannian manifolds.
For example, to determine the least dimensional Euclidean space into which a given 
Riemannian manifold can be locally isometrically embedded is a quite natural problem in 
Riemannian submanifold theory.
However, this problem is still unsolved except for some special manifolds. 
Concerning this problem, the 
first author and Kaneda have studied deeply for the case of Riemannian 
symmetric spaces. 
They verified that some of Kobayashi's standard embeddings (\cite{Kobayashi}) of symmetric $R$-spaces 
give the least dimensional isometric embedding even in the local standpoint.
Moreover they proved that some of them are locally rigid (for details, see \cite{AK10} and 
references therein). 

In the present paper, we study local isometric embeddings of three-dimensional Lie groups $G$ 
equipped with a left-invariant Riemannian metric $g$ into the four-dimensional Euclidean space $\mathbb{R}^4$.

We denote by $\g$ the Lie algebra of $G$ and $\naiseki$ the inner product of $\g$ determined by 
the left invariant Riemannian metric $g$ at the unit element $e \in G$.
Since the existence or non-existence of local isometric embeddings of $(G,g)$ is 
completely determined by its infinitesimal data $(\g,\naiseki)$, we may state the results in 
terms of Lie algebras.
We often call the pair $(\g,\naiseki)$ the metric Lie algebra.
Before stating the main results, we first list the classification table of three-dimensional real Lie algebras:

\begin{table}[ht]
\begin{tabular}{|c|l|c|} \hline
Name \rule[-0.3cm]{0cm}{0.8cm} & \hspace{1cm} Non-zero bracket relation & \\ \hline
$\mathbb{R}^3$ \rule{0cm}{0.5cm} & & Abelian\\
$\mathfrak{h}_3$ \rule{0cm}{0.5cm} & $[e_1,e_2]=e_3$ & Nilpotent\\
$\mathfrak{r}_3$ \rule{0cm}{0.5cm} & $[e_1,e_2]=e_2+e_3,~[e_1,e_3]=e_3$ & Solvable\\
$\mathfrak{r}_{3, \alpha}$ \rule{0cm}{0.5cm} & $[e_1,e_2]=e_2,~[e_1,e_3]=\alpha e_3$ \hspace{3mm} $(-1 \leq \alpha \leq 1)$ & Solvable \\
$\mathfrak{r}^{\prime}_{3, \alpha}$ \rule{0cm}{0.5cm} & $[e_1,e_2]=\alpha e_2 - e_3,~[e_1,e_3]=e_2 + \alpha e_3$ \hspace{3mm} 
$(\alpha \geq 0)$ & Solvable \\ 
$\mathfrak{sl}_2(\mathbb{R})$ \rule{0cm}{0.5cm} & $[e_1,e_2]=-e_3,~[e_2,e_3]=e_1,~[e_3,e_1]=e_2$ & Simple  \\
$\mathfrak{so}(3)$ \rule[-0.3cm]{0cm}{0.8cm} & $[e_1,e_2]=e_3,~[e_2,e_3]=e_1,~[e_3,e_1]=e_2$ & Simple \\
\hline
\end{tabular}  
\vspace{0.3cm}
\label{table:three-dimensional}
\caption{Three-dimensional real Lie algebras} 
\end{table} 

\noindent
Then our main results are summarized in the following:

\begin{thm}\label{Main}
Let $G$ be a three-dimensional Lie group endowed with a left-invariant Riemannian metric $g$. 
If $(G, g)$ can be locally isometrically embedded into $\mathbb{R}^4$, 
then the corresponding metric Lie algebra $(\g, \naiseki)$ is isometric to one of the following 
spaces up to scaling:
\begin{itemize}
\item Abelian Lie algebra $\mathbb{R}^3$ with any inner product.
\item $\mathfrak{r}_{3, 0}$ with the inner product $\naiseki$.
\item $\mathfrak{r}^{\prime}_{3, 0}$ with the inner product $\naiseki$.
\item $\mathfrak{so}(3)$ with the inner product $\naiseki$.
\end{itemize}
Here $\naiseki$ for each Lie algebra is the inner product such that the basis $\{ e_1, e_2, e_3 \}$ in 
Table~$1$ is orthonormal.
\end{thm}

Note that among the above four metric Lie algebras, the first and the third are flat.
($\mathfrak{r}^{\prime}_{3, 0}$ is isomorphic to the Lie algebra of infinitesimal 
motions of three-dimensional Euclidean space.)
The second metric Lie algebra $(\mathfrak{r}_{3, 0}, \naiseki)$ corresponds to the direct 
product of the hyperbolic plane $\mathbb{R}\mathrm{H}^2$ and the one-dimensional flat Euclidean 
space $\mathbb{R}^1$.
The fourth metric Lie algebra $(\mathfrak{so}(3), \naiseki)$ corresponds to the space of constant positive 
sectional curvature.
It is well known that these four spaces can be locally isometrically embedded into $\mathbb{R}^4$, 
and our task is to show the non-existence of local isometric embeddings into $\mathbb{R}^4$ for 
the remaining three-dimensional metric Lie algebras $(\g,\naiseki)$.

However, this is not an easy task.
The Gauss equation is a well known obstruction to the existence of local isometric embeddings of 
Riemannian manifolds.
For some cases, it gives a powerful tool to show the non-existence of local isometric embeddings 
(see \cite{A1}, \cite{AK10}, \cite{Bor}, \cite{Mas1}, \cite{Mas2}, \cite{R1}, \cite{R2} 
and references in \cite{AK10}).

In our problem, however, it does not give a sufficient tool to prove the main theorem.
In fact, there occur the cases where the Gauss equation admits a solution in codimension $1$, though 
the corresponding metric Lie algebras cannot be locally isometrically embedded into $\mathbb{R}^4$.
To show the non-existence of embeddings for such spaces, we must consider a higher order obstruction, 
which we call the \textit{derived Gauss equation} in this paper (for details, see Section~2).

This higher order obstruction was introduced for the first time by Kaneda (\cite[p.197]{K1}), though it 
does not serve as an obstruction when a Riemannian manifold is locally symmetric (for detail, see Section~5).

As a result, metric Lie algebras $(\g,\naiseki)$ not listed in Theorem~\ref{Main} are divided into two classes:
One is the class, not admitting a solution of the Gauss equation.
The other is the class, admitting a solution of the Gauss equation, but does not admit a solution of 
the derived Gauss equation.
The difference of these classes are explicitly examined in Sections~4 and 5.

To complete our classification, we must know all left invariant Riemannian metrics on a given 
three-dimensional Lie group in advance.
Fortunately, this problem is already settled by 
the second author and Tamaru (\cite{HT}), 
by considering Milnor-type theorems.
Applying their results, we show the non-existence of the solution of the Gauss equation or 
the derived Gauss equation for the remaining cases individually.
We may say that our results are a good application of Milnor-type theorems.

By Janet-Cartan's theorem we already know that any three-dimensional metric Lie algebra 
$(\g,\naiseki)$ can be locally isometrically embedded into the six-dimensional Euclidean space 
$\mathbb{R}^6$.
Therefore, our next problem is to consider the existence or non-existence of local isometric embeddings 
into the five-dimensional Euclidean space $\mathbb{R}^5$ for the spaces not listed in Theorem~\ref{Main}.
We will treat this problem in the forthcoming papers.

The contents of this paper is as follows.
In Section~2, we first recall some fundamental facts on local isometric embeddings, following Kaneda 
and Tanaka (\cite{KT}), Kaneda (\cite{K1}). 
We emphasize that the Gauss equation and the derived Gauss equation are the consequences of the integrability 
conditions of differential equations of local isometric embeddings.
In Section~3, we introduce Milnor frames and Milnor-type theorems, and review the results for 
the case of three-dimensional Lie algebras.
In Sections~4 and 5 we discuss the solvability of the Gauss equation and the derived Gauss equation for 
each metric Lie algebra $(\g,\naiseki)$.

\bigskip
 
{\bf Acknowledgement.}
The authors would like to thank Eiji Kaneda and Hiroshi Tamaru for helpful comments and valuable discussions. 
The first author was supported by JSPS KAKENHI Grant Number 16K05132, and the second author was supported 
by JSPS KAKENHI Grant Number 16K17063.

\section{Differential equations associated with isometric embeddings}
In this section, we recall fundamental facts on local isometric embeddings, 
and introduce a framework for studying isometric embedding problems in terms of systems of differential 
equations following \cite{K1, KT}.

Let $(M,g)$ be an $n$-dimensional Riemannian manifold, $\nabla$ be the covariant differentiation 
associated with the Levi-Civita connection of $(M,g)$, and $R$ be the curvature tensor field of type $(1,3)$ 
on $(M,g)$. 
For a $C^{\infty}$ function $f$ on $M$, and tangent vectors $x_1, \ldots, x_k $ at an arbitrary point 
$p$ of $M$, we define a scalar $\nabla^k_{(x_1, \ldots, x_k)} f \in \mathbb{R}$ by     
$\nabla^k_{(x_1, \ldots, x_k)} 
f := (\nabla \cdots \nabla) f (x_1, \ldots, x_k)$ where $(\nabla \cdots \nabla) f$ is the 
\textit{k-th covariant derivative of} $f$.
Let $F=(f^1, \ldots, f^m)$ be a $C^{\infty}$-mapping of $M$ into the $m$-dimensional Euclidean space 
$\mathbb{R}^m$.
We define the \textit{$\mathit{k}$-th covariant derivative} of $F$ by 
\begin{align*}
\nabla_{(x_1, \cdots, x_k)}^k F:=(\ldots, \nabla_{(x_1, \cdots, x_k)}^k f^i, \ldots ) \in \mathbb{R}^m.  
\end{align*}
Then the following equations hold:
\begin{align}
\label{int_1} \nabla_{(x_1, x_2)}^2 F & = \nabla_{(x_2, x_1)}^2 F,\\ 
\label{int_2} \nabla_{(x_1, x_2, x_3)}^3 F & = \nabla_{(x_1, x_3, x_2)}^3 F, \\
\label{int_3} \nabla_{(x_1, x_2, x_3)}^3 F & = \nabla_{(x_2, x_1, x_3)}^3 F - \nabla_{R(x_1, x_2)x_3}^1 F.
\end{align}
The above conditions are called the \textit{integrability conditions}.
In particular, the last equation~(\ref{int_3}) is called the \textit{Ricci formula}.
 
Let $\naiseki_{\mathbb{R}^m}$ denote the standard inner product on the $m$-dimensional Euclidean 
space $\mathbb{R}^m$.
An embedding $F=(f^1, \ldots , f^m): M \rightarrow \mathbb{R}^m$ is called an \textit{isometric embedding} if 
\begin{align*}
\langle \mathit{d}F, \mathit{d}F \rangle_{\mathbb{R}^m} = g
\end{align*}
holds.
Clearly the inequality $m \geq n$ holds if an isometric embedding $F$ exists.

In this paper we consider only ``local" isometric embeddings, i.e., 
we assume that an embedding $F$ is defined only on a sufficiently small open neighborhood 
of a given point $p$ of $M$.

\begin{thm}[\cite{K1, KT}]\label{TK}
Assume that $F$ is a local isometric embedding of $(M,g)$ into $\mathbb{R}^m$. 
Then, for each point $p \in M$, and any vectors $x_1, x_2, x_3,x_4, x_5 \in T_p M$, 
$F$ satisfies the following:
\begin{enumerate}
\item \rule{0cm}{0.5cm} $\langle \nabla_{x_1}^1 F, \ \nabla_{x_2}^1 F \rangle_{\mathbb{R}^m} = g(x_1,x_2)$,
\item \rule{0cm}{0.5cm} $\langle \nabla_{(x_1, x_2)}^2 F, \ \nabla_{x_3}^1 F \rangle_{\mathbb{R}^m} = 0$,
\item \rule{0cm}{0.5cm} $\langle \nabla_{(x_1, x_2, x_3)}^3 F, \ \nabla_{x_4}^1 F \rangle_{\mathbb{R}^m} 
+ \langle \nabla_{(x_2, x_3)}^2 F, \ \nabla_{(x_1, x_4)}^2 F \rangle_{\mathbb{R}^m} = 0$,
\item \rule{0cm}{0.5cm} $\langle \nabla_{(x_1, x_3)}^2 F, \ \nabla_{(x_2, x_4)}^2 F \rangle_{\mathbb{R}^m} 
- \langle \nabla_{(x_1, x_4)}^2 F, \ \nabla_{(x_2, x_3)}^2 F \rangle_{\mathbb{R}^m} 
= R(x_1, x_2, x_3, x_4)$,
\item \rule{0cm}{0.5cm} $\langle \nabla_{(x_5, x_1, x_3)}^3 F, \ \nabla_{(x_2, x_4)}^2 F \rangle _{\mathbb{R}^m}
+ \langle \nabla_{(x_1, x_3)}^2 F, \ \nabla_{(x_5, x_2, x_4)}^3 F \rangle_{\mathbb{R}^m} \\
\rule{0cm}{0.5cm} \quad - \langle \nabla_{(x_5, x_1, x_4)}^3 F, \ \nabla_{(x_2, x_3)}^2 F \rangle_{\mathbb{R}^m} 
- \langle \nabla_{(x_1, x_4)}^2 F, \ \nabla_{(x_5, x_2, x_3)}^3 F \rangle_{\mathbb{R}^m} \\
\rule{0cm}{0.5cm} \quad \quad = (\nabla R)(x_5,x_1, x_2, x_3, x_4)$.
\end{enumerate} 
\end{thm}
\noindent
Note that $R$ and $\nabla R$ are the Riemannian curvature tensor of type $(0,4)$ and its covariant derivative 
respectively, which will be defined in the next section.
Recall that $\nabla^2 F$ is called the \textit{second fundamental form},
the equation (4) is the so-called \textit{Gauss equation} for the isometric embedding. 
In this paper we call the equation (5) the \textit{derived Gauss equation}.
The Gauss equation and the derived Gauss equation serve as obstructions to the existence of 
local isometric embeddings. 
The proof of our main theorem is given by investigating the solvability of the equations (4) and (5). 
Note that, by taking the covariant derivatives $\nabla^k F$ of $F$ for large $k$, we can theoretically obtain 
many other higher order obstructions.

\section{Preliminaries on curvatures and Milnor-type theorems}
In order to investigate the existence or non-existence of solutions of the Gauss equations and 
the derived Gauss equations, we calculate Riemannian curvature tensors $R$ and their covariant derivatives 
$\nabla R$ for all three-dimensional Lie groups equipped with any left-invariant Riemannian metric.
In this section, we calculate the curvatures and recall Milnor-type theorems for three-dimensional Lie algebras.

\subsection{Curvatures of three-dimensional Lie groups}
In this subsection, we calculate the curvatures of three-dimensional Lie groups equipped with 
a left-invariant Riemannian metric. 
Let $(G,g)$ be a simply-connected Lie group endowed with a left-invariant Riemannian metric, 
and $(\g, \naiseki)$ be the corresponding metric Lie algebra of $(G,g)$. 
We will discuss curvatures of $(G,g)$ by the metric Lie algebra $(\g, \naiseki)$ 
since bracket products and covariant derivatives can be calculated by metric Lie algebra 
under the identification of left-invariant vector fields of $(G,g)$ and the Lie algebra $\g$. 

Let us denote by $\{ e_1, e_2, e_3\}$ a basis of the Lie algebra $\g$. 
We identify each $e_i$ with a left-invariant vector field of $G$.
The Levi-Civita connection 
$\nabla: \g \times \g \longrightarrow \g$ is given by
\begin{align*}
2 \langle \nabla _{e_i} e_j,e_k \rangle = \langle [e_k,e_i],e_j \rangle + \langle e_i,[e_k,e_j] 
\rangle +\langle [e_i,e_j], e_k \rangle .
\end{align*}
We define the Riemannian curvature tensor $R$ of type $(1,3)$ by
\begin{align*}
R(e_i,e_j)e_k:=\nabla _{e_i} \nabla _{e_j} e_k - \nabla _{e_j} \nabla _{e_i} e_k - \nabla _{[e_i,e_j]} e_k.
\end{align*}
We also denote by $R$ the Riemannian curvature tensor of type $(0,4)$ which is given by
\begin{align*}
R (e_i,e_j,e_k,e_l):=- \langle R(e_i,e_j)e_k, e_l \rangle .
\end{align*}
The covariant derivative $\nabla R$ of $R$ is given by 
\begin{align*}
(\nabla R) (e_p,e_i,e_j,e_k,e_l):=&-R (\nabla _{e_p} e_i,e_j,e_k,e_l)-R (e_i,\nabla _{e_p} e_j,e_k,e_l)\\
&-R (e_i,e_j,\nabla _{e_p} e_k,e_l)-R (e_i,e_j,e_k,\nabla _{e_p} e_l).	
\end{align*}
Hereinafter, we write $R_{ijkl}:=R (e_i,e_j,e_k,e_l)$, and 
$\nabla_p R_{ijkl}:=(\nabla R) (e_p,e_i,e_j,e_k,e_l)$ for short.

\begin{lem}\label{R}
Let $\g$ be a three-dimensional Lie algebra, and $\naiseki$ be an inner product on $\g$. 
Suppose that there exist $a,b,c,d,t \in \mathbb{R}$ satisfying $(a+d)t=0$ and 
an orthonormal basis 
$\{ e_1, e_2, e_3 \}$ 
of $\g$ with respect to $\naiseki$ 
such that the bracket relations are given by 
\begin{align*} 
[e_1, e_2] = a e_2 + 2b e_3, \quad [e_1, e_3] = 2c e_2 + d e_3, \quad [e_2, e_3]=2t e_1.
\end{align*} 
Then, the Riemannian curvature tensors of type $(0,4)$ are given by
\begin{align*}
& R_{1212} 
= - (a^2 + 3b^2 -c^2+2b c-t^2-2 (b+c)t), \\
& R_{1313}  
= - (-b^2 +3c^2+d^2+2b c-t^2+2(b+c)t), \\ 
& R_{2323}  
= b^2 +c^2 -a d +2b c-3t^2+2(b-c)t ,\\
& R_{1213}  
= - 2(a c+b d+ a t), \\ 
& R_{1223}  
= 0, \\ 
& R_{1323}  
= 0.
\end{align*}  
Here the condition $(a+d)t=0$ is necessary in order that $\g$ is a Lie algebra.
\end{lem}

\begin{proof}
We first calculate the Levi-Civita connection $\nabla$. 
By direct calculations, we see that 
\begin{align*} 
\nabla_{e_1} e_1 &= 0 , \quad 
\nabla_{e_1} e_2 = (b-c-t) e_3 , \quad 
\nabla_{e_1} e_3 = -(b-c-t) e_2 , \\ 
\nabla_{e_2} e_1 &= - a e_2 - (b+c+t) e_3 ,\quad 
\nabla_{e_2} e_2 = a e_1 , \quad 
\nabla_{e_2} e_3 = (b+c+t) e_1 , \\
\nabla_{e_3} e_1 &= - (b+c-t) e_2 - d e_3 ,\quad 
\nabla_{e_3} e_2 = (b+c-t) e_1 ,\quad 
\nabla_{e_3} e_3 = d e_1 .
\end{align*} 

\noindent
One can thus calculate the Riemannian curvatures $R$. 
For $R(e_1, e_2) e_1$, 
we have 
\begin{align*} 
R(e_1, e_2) e_1 &= (a^2 + 3b^2 -c^2+2b c-t^2-2 (b+c)t) e_2 + 2(ac+bd+at) e_3 ,
\end{align*} 
and this yields that
\begin{align*}
  \left\{
    \begin{array}{ll}
R_{1212} 
&= -(a^2 + 3b^2 -c^2+2b c-t^2-2 (b+c)t), \\
R_{1213} 
&= -2(ac+bd+at).
    \end{array}
  \right.
\end{align*}
Similarly, we obtain 
\begin{align*}
R(e_1, e_3) e_1 &= 2(ac+bd-dt) e_2 + (-b^2 +3c^2+d^2+2b c-t^2+2(b+c)t) e_3, \\ 
R(e_2, e_3) e_2 &= -(b^2 +c^2-a d+2b c-3t^2+2(b-c)t) e_3,\\
R(e_1, e_3) e_2 &= -2(a c+b d-d t) e_1,  
\end{align*} 
here we use the condition $at=-dt$.  
We complete the proof.
\end{proof} 

Next we calculate $\nabla R$ for some cases.
By direct calculations, we obtain the following Lemmas.

\begin{lem}\label{R1}
Let $\g$ be a three-dimensional solvable Lie algebra, and $\naiseki$ be an inner product on $\g$. 
Suppose that there exist $a,b,c,d \in \mathbb{R}$ and 
an orthonormal basis 
$\{ e_1, e_2, e_3 \}$ 
of $\g$ with respect to $\naiseki$ 
such that the bracket relations are given by 
\begin{align*} 
[e_1, e_2] = a e_2 + 2b e_3, \quad [e_1, e_3] = 2c e_2 + d e_3.
\end{align*} 
Then, we have 
\begin{align*}
\nabla_{2} R_{1223} & = (b+c)(R_{1212}-R_{2323}) -aR_{1213},\\
\nabla_{2} R_{3123} & = a(R_{1313}-R_{2323}) -(b+c)R_{1213},\\
\nabla_{3} R_{1223} & = d(R_{1212}-R_{2323}) -(b+c)R_{1213},\\
\nabla_{3} R_{3123} & = (b+c)(R_{1313}-R_{2323}) -dR_{1213}.
\end{align*}  
\end{lem}

\begin{lem}\label{R2}
Let $\g$ be a three-dimensional simple Lie algebra, and $\naiseki$ be an inner product on $\g$. 
Suppose that there exist $u,v \in \mathbb{R}$ and 
an orthonormal basis 
$\{ e_1, e_2, e_3 \}$ 
of $\g$ with respect to $\naiseki$ 
such that the bracket relations are given by 
\begin{align*} 
[e_1, e_2] = \frac{-u+v}{2} \, e_3, \quad [e_3, e_1] = \frac{u+v}{2} \, e_2, \quad [e_2, e_3] = e_1.
\end{align*} 
Then, we have 
\begin{align*}
\nabla_{1} R_{1213} & = \frac{1}{2}u(v-1)^2,\\
\nabla_{2} R_{1223} & = \frac{1}{4}(u-1)^2(u-v+2),\\
\nabla_{3} R_{2313} & = -\frac{1}{4}(u+1)^2(u+v-2).\\
\end{align*}  
\end{lem}

\noindent
These Lemmas~3.2 and 3.3 are used in Section~$5$.

\subsection{Milnor frames and Milnor-type theorems}
In this subsection, we recall Milnor frames and Milnor-type theorems.
They are quite powerful tools to study Riemannian geometry of Lie groups equipped with 
left-invariant Riemannian metrics, also well applied to the problem of local isometric embeddings of Lie groups. 
We refer to \cite{Mil} for a deeper discussion of Milnor frames and \cite{HTT} for Milnor-type theorems.

First of all, we review Milnor's result. 
\begin{thm}[\cite{Mil}]\label{Mi}
Let $\g$ be a three-dimensional unimodular Lie algebra. 
For every inner product $\naiseki$ on $\g$, 
there exist an orthonormal basis $\{ e_1, e_2, e_3 \}$ with respect to $\naiseki$, and real numbers  
$\lambda_1, \lambda_2, \lambda_3 $ such that the bracket relations are given by 
\begin{align*}
[e_1, e_2]=\lambda_3 e_3, \ [e_2, e_3]=\lambda_1 e_1, \ [e_3, e_1]=\lambda_2 e_2.
\end{align*}
\end{thm}

\noindent
The above bases are nowadays called the \textit{Milnor frames}.
Note that, for each three-dimensional unimodular Lie algebra $\g$, the possible signatures of 
$\lambda_1, \lambda_2, \lambda_3$ are already determined (see \cite{Mil}). 
For simple cases, possible signatures are $(\lambda_1, \lambda_2, \lambda_3) = (+,+,-)$ when 
$\g=\mathfrak{sl}_2(\mathbb{R})$, and $(\lambda_1, \lambda_2, \lambda_3) =(+,+,+)$ when 
$\g=\mathfrak{so}(3)$.
From Theorem~\ref{Mi}, for each three-dimensional unimodular Lie algebra, any inner product 
on $\g$ can be examined by using at most three-parameters $\lambda_1, \lambda_2, \lambda_3$.

When $\g$ is simple, by considering ``up to scaling", we may assume $\lambda_1=1$ as follows.  
\begin{prop}\label{Ms}
\begin{enumerate} 
\item For every inner product $\naiseki$ on $\mathfrak{sl}_2(\mathbb{R})$, 
there exist $\lambda_2 >0$, $\lambda_3 <0$, $k > 0$, 
and an orthonormal basis $\{x_1, x_2, x_3\}$ of $\mathfrak{sl}_2(\mathbb{R})$ with respect to $k \naiseki$ 
such that the bracket relations are given by 
\begin{align*}
[x_1, x_2]=\lambda_3 x_3, \quad [x_2, x_3]=x_1, \quad [x_3, x_1]=\lambda_2 x_2.
\end{align*}
\item For every inner product $\naiseki$ on $\mathfrak{so}(3)$, 
there exist $\lambda_2, \lambda_3 >0$, $k > 0$, 
and an orthonormal basis $\{x_1, x_2, x_3\}$ of $\mathfrak{so}(3)$ with respect to $k \naiseki$ 
such that the bracket relations are given by 
\begin{align*}
[x_1, x_2]=\lambda_3 x_3, \quad [x_2, x_3]=x_1, \quad [x_3, x_1]=\lambda_2 x_2.
\end{align*}
\end{enumerate}
\end{prop}

\begin{proof}
We only prove (1). 
Take any inner product on $\mathfrak{sl}_2(\mathbb{R})$.
By Theorem~\ref{Mi}, there exist real numbers  
$\lambda^{\prime}_1, \lambda^{\prime}_2 >0$, $\lambda^{\prime}_3<0$, and an orthonormal basis 
$\{ e_1, e_2, e_3 \}$ with respect to $\naiseki$ such that the bracket relations are given by 
\begin{align*}
[e_1, e_2]=\lambda^{\prime}_3 e_3, \quad [e_2, e_3]=\lambda^{\prime}_1 e_1, \quad [e_3, e_1]=\lambda^{\prime}_2 e_2.
\end{align*}
By putting $\lambda_2:= (\lambda^{\prime}_2/\lambda^{\prime}_1)>0$, $\lambda_3:= 
(\lambda^{\prime}_3/\lambda^{\prime}_1)<0$, 
$k:=(\lambda^{\prime}_1)^2 >0$ and $\{ x_1, x_2, x_3 \}:= \{ (1/\lambda^{\prime}_1)e_1, 
(1/\lambda^{\prime}_1)e_2, (1/\lambda^{\prime}_1)e_3 \}$, 
one can easily see that $\{ x_1, x_2, x_3 \}$ is an orthonormal basis of $\mathfrak{sl}_2(\mathbb{R})$ 
with respect to $k\naiseki$.
By direct calculations, one has bracket relations as follows:
\begin{align*}
[x_1, x_2]&=[(1/\lambda^{\prime}_1)e_1, (1/\lambda^{\prime}_1)e_2]
=(1/\lambda^{\prime}_1)^2 \lambda^{\prime}_3 e_3 =\lambda_3 x_3, \\
[x_2, x_3]&= [(1/\lambda^{\prime}_1)e_2, (1/\lambda^{\prime}_1)e_3]=(1/\lambda^{\prime}_1)^2 \lambda^{\prime}_1 e_1=x_1, \\
[x_3, x_1]&=[(1/\lambda^{\prime}_1)e_3, (1/\lambda^{\prime}_1)e_1]=(1/\lambda^{\prime}_1)^2 \lambda^{\prime}_2 e_2=\lambda_2 x_2.
\end{align*}

\noindent
We complete the proof of (1).
The statement (2) can be verified in the same way.
\end{proof}

As for the three-dimensional solvable Lie algebras, all inner products can be studied by using 
at most one-parameter. 
In \cite{HTT}, the second author, Tamaru and Terada gave a procedure to construct an analogue 
of Milnor frames for any Lie algebra, which is called Milnor-type theorems.
The basic idea of Milnor-type theorems is based on the study of the moduli space of all 
left-invariant Riemannian metrics on a given Lie group (see \cite{KTT}).

For some time let $G$ be an $n$-dimensional Lie group, and $\g$ be its Lie algebra. 
We denote by $\M$ the set of all left-invariant Riemannian metrics on $G$, 
which can be naturally identified with 
\begin{align*}
\M := \{ \naiseki \mid \mbox{an inner product on $\g$} \} .  
\end{align*}
We identify $\g$ with $\mathbb{R}^n$ as vector spaces from now on.
Then, since $\GL$ acts transitively on $\M$ by 
\begin{align*}
g.\langle \cdot, \cdot \rangle := \langle g^{-1} (\cdot), g^{-1} (\cdot) \rangle , 
\end{align*}
we have an identification 
\begin{align*}
\M \cong \GL / \OO . 
\end{align*}
We define an equivalence relation ``isometric up to scaling" on $\M$.
\begin{defi}\label{I}
Two inner products $\naiseki_1$ and $\naiseki_2$ on $\g$ 
are said to be \textit{isometric up to scaling} 
if there exist $k > 0$ and an automorphism 
$f: \g \rightarrow \g$ 
such that $\langle \cdot , \cdot \rangle_1 = k \langle f (\cdot) , f (\cdot) \rangle_2$.  
\end{defi}
We will denote by $[\naiseki]$ the equivalence class of the inner product $\naiseki$.
We call the quotient space of $\M$ by the equivalence relation ``isometric up to scaling" the 
\textit{moduli space of left-invariant Riemannian metrics} and express it as 
\begin{align*}
\PM(\g) &:= \A \backslash \M \cong \A \backslash \GL / \OO \\
&= \{ [\naiseki] \mid \naiseki \in \M \}.
\end{align*}
Here 
$$
\A := \{ (c \cdot \mathrm{id}) \varphi \in \GL \mid c \in \R \backslash \{ 0 \}, \varphi \in \Aut \}.
$$
By considering the set of representatives of $\PM(\g)$, one can obtain an analogue of Milnor frames. 
For each Lie group, a procedure to construct them is called the Milnor-type theorem.
For a deeper discussion of Milnor-type theorems, we refer to \cite{HTT}. 

\begin{rem}
Let $G$ be a simply-connected Lie group with Lie algebra $\g$.
If two inner products $\naiseki_1$ and $\naiseki_2$ on $\g$ are isometric up to scaling in the sense of 
Definition~\ref{I}, then the correspondence left-invariant Riemannian metrics $g_1$ and $g_2$ on $G$ are 
also isometric up to scaling as Riemannian metrics. For details, we refer to \cite[Remark2.3]{KTT}.
\end{rem}

Now we return to the three-dimensional case.
For each three-dimensional solvable Lie algebra in Table~$1$, we recall Milnor-type theorems. 
We summarize the results in the following proposition.
For more details, we refer to \cite{HT}.

\begin{prop}[\cite{HT}]\label{M1}
\begin{enumerate}
\item For every inner product $\naiseki$ on $\mathfrak{r}_3$, 
there exist $\lambda > 0$, $k > 0$ and an orthonormal basis $\{x_1, x_2, x_3\}$ of $\mathfrak{r}_3$ 
with respect to $k \naiseki$ 
such that the non-zero bracket relations are given by 
\begin{align*}
[x_1,x_2] = x_2 + 2 \lambda x_3 , \quad [x_1,x_3] = x_3 . 
\end{align*} 
\item For every inner product $\naiseki$ on $\mathfrak{r}_{3, \alpha}$ $(-1 \leq \alpha <1)$, 
there exist $\lambda \in \mathbb{R}$, $k > 0$, 
and an orthonormal basis $\{x_1, x_2, x_3\}$ of $\mathfrak{r}_{3, \alpha}$ with respect to $k \naiseki$ 
such that 
the non-zero bracket relations are given by 
\begin{align*}
[x_1 , x_2] = x_2 + 2 \lambda (\alpha-1)x_3 , \quad [x_1 , x_3] = \alpha x_3 . 
\end{align*} 
\item For every inner product $\naiseki$ on $\mathfrak{r}^{\prime}_{3, \alpha}$ $(\alpha \geq 0)$, 
there exist $\lambda \geq 1$, $k > 0$, 
and an orthonormal basis $\{x_1, x_2, x_3\}$ of $\mathfrak{r}^{\prime}_{3, \alpha}$ with respect to $k \naiseki$ 
such that  
the non-zero bracket relations are given by 
\begin{align*}
[x_1,x_2] = \alpha x_2 - \lambda x_3 , \quad [x_1,x_3] = \frac{1}{\lambda} x_2 + \alpha x_3 . 
\end{align*} 
\end{enumerate}
\end{prop}

\begin{rem}
When $\g$ is isomorphic to $\mathbb{R}^3$, $\mathfrak{h}_3$ or $\mathfrak{r}_{3, 1}$, 
the inner product on $\g$ is unique up to isometry and scaling 
(see \cite{KTT, Lau03}).
We also remark that we can apply Lemma~\ref{R} to the above solvable cases by putting $t=0$ in 
Lemma~\ref{R}.
\end{rem}

\section{The Gauss equation}
Let $(G,g)$ be a three-dimensional Lie group equipped with a left-invariant Riemannian metric, 
and $(\g, \naiseki)$ be the corresponding metric Lie algebra of $(G,g)$. 
In the following arguments we only consider local isometric embeddings $F$ defined in 
a sufficiently small neighborhood of the unit element $e$ of $G$ into the four-dimensional Euclidean 
space $\mathbb{R}^4$.
Here and hereinafter, we will symbolically write this embedding as $F: (\g, \naiseki) \rightarrow 
\mathbb{R}^4$.
We remark that the existence or non-existence of local isometric embeddings of left-invariant 
Riemannian metrics is completely determined by the infinitesimal character of $(G,g)$.
In addition, also remark that the existence or non-existence is unchanged if we replace the inner product $\naiseki$ 
with $k \naiseki$ for any $k>0$.

In this section, as a first step, we investigate the solutions of the Gauss equation in codimension $1$. 
 
\subsection{Preliminary}

Let $(\g, \naiseki)$ be a three-dimensional metric Lie algebra with an orthonormal basis 
$\{ x_1, x_2, x_3 \}$, and $F$ be a local isometric embedding of $(\g,\naiseki)$ into $\mathbb{R}^4$. 
Recall that $F$ satisfies the Gauss equation:
\begin{align*}
\langle \nabla^2_{(x_i,x_k)} F, \nabla^2_{(x_j,x_l)} F \rangle_{\mathbb{R}^4} 
- \langle \nabla^2_{(x_i,x_l)} F, \nabla^2_{(x_j,x_k)} F \rangle_{\mathbb{R}^4} 
= R (x_i, x_j, x_k, x_l)
\end{align*}
for each $i$, $j$, $k$, $l \in \{ 1,2,3 \}$.

Let $N$ be a unit normal vector of the embedding $F$ at $e$.
The normal space of $F$ at $e$ is spanned by $N$.
Recall that $(\nabla ^2 F)_e$ is the second fundamental form of $F$ at $e$. 
Put $h_{ij}:=\langle (\nabla^2_{(x_i, x_j)} F)_e, N \rangle$.
Then from the integrability condition~(\ref{int_1}) we have $h_{ij} = h_{ji}$.
Since the tangential part of $(\nabla^2_{(x_i, x_j)} F)_e$ is $0$, the Gauss equation in codimension $1$ 
at $e$ can be expressed as 
\begin{eqnarray}
\left\{
\begin{array}{l}\label{Gauss_eq}
h_{11}h_{22}-h_{12}^2=R_{1212}, \\
h_{11}h_{33}-h_{13}^2=R_{1313}, \\
h_{22}h_{33}-h_{23}^2=R_{2323}, \\
h_{11}h_{23}-h_{13}h_{12}=R_{1213}, \\
h_{12}h_{23}-h_{13}h_{22}=R_{1223}, \\
h_{12}h_{33}-h_{13}h_{23}=R_{1323}. 
\end{array}
\right.
\end{eqnarray}

In the following arguments we consider (\ref{Gauss_eq}) as abstract algebraic equations on 
$h_{ij}$ for a given curvature tensor $R$.
Namely, we do not assume the existence of local isometric embeddings $F$ in advance.
The equation (\ref{Gauss_eq}) serves as an obstruction to the existence of local isometric embeddings $F$.
In fact, if we can show the non-existence of a solution of (\ref{Gauss_eq}) for a given $R$, we can conclude that 
$(\g, \naiseki)$ does not admit a local isometric embedding into $\mathbb{R}^4$.
For left-invariant Riemannian metrics, the solvability of the Gauss equation does not depend on the choice of 
reference points, and hence we have only to consider the Gauss equation (\ref{Gauss_eq}) at $e$, 
concerning its solvability.

The Gauss equation in codimension $1$ for general $n$-dimensional Riemannian manifolds was deeply investigated 
by Weise \cite{W} and Thomas \cite{T} for $n \geq 3$.
In fact they have shown the following theorem concerning the solvability and the inverse 
formula of the Gauss equation in codimension $1$.
For our purpose, we here state their results for only three-dimensional case.

\begin{thm}[\cite{W}, \cite{T}]\label{Thomas}
Let $(M,g)$ be a three-dimensional Riemannian manifold.
If the Gauss equation {\rm (\ref{Gauss_eq})} admits a solution, then the following inequality holds:
\begin{align*}
T:=\begin{vmatrix}
R_{1212} & R_{1213} & R_{1223}\\
R_{1213} & R_{1313} & R_{1323}\\
R_{1223} & R_{1323} & R_{2323}
\end{vmatrix} \geq 0. 
\end{align*}
Conversely, if the inequality $T>0$ holds, then the Gauss equation {\rm (\ref{Gauss_eq})} has a solution.
Moreover its solution is uniquely determined up to sign explicitly in the following form $(\varepsilon= \pm 1)$:
\begin{align*}
h_{11} = \frac{\varepsilon}{\sqrt{T}}
\begin{vmatrix}
R_{1212} & R_{1213} \\
R_{1312} & R_{1313} 
\end{vmatrix}, \quad 
h_{12} = \frac{\varepsilon}{\sqrt{T}}
\begin{vmatrix}
R_{1212} & R_{1223} \\
R_{1312} & R_{1323} 
\end{vmatrix}, \\
\rule{0cm}{0.9cm} h_{13} = \frac{\varepsilon}{\sqrt{T}}
\begin{vmatrix}
R_{1213} & R_{1223} \\
R_{1313} & R_{1323} 
\end{vmatrix}, \quad 
h_{22} = \frac{\varepsilon}{\sqrt{T}}
\begin{vmatrix}
R_{1212} & R_{1223} \\
R_{2312} & R_{2323} 
\end{vmatrix}, \\
\rule{0cm}{0.9cm} h_{23} = \frac{\varepsilon}{\sqrt{T}}
\begin{vmatrix}
R_{1213} & R_{1223} \\
R_{2313} & R_{2323} 
\end{vmatrix}, \quad 
h_{33} = \frac{\varepsilon}{\sqrt{T}}
\begin{vmatrix}
R_{1313} & R_{1323} \\
R_{2313} & R_{2323} 
\end{vmatrix}. \\
\end{align*}
\end{thm}

\noindent
For Theorem~\ref{Thomas}, we refer to \cite[p.530$\sim$531]{W}, \cite[p.192, p.199]{T}, 
\cite[p.42$\sim$43]{Kawa} and \cite[p.132]{A1}.

Concerning the solvability of the Gauss equation in codimension $1$ for general $3$-dimensional 
Riemannian manifolds, Jacobowitz \cite{J1} showed the following result:``Non flat curvature tensors 
$R_{ijkl}$ admit a solution of the Gauss equation in codimension $1$ if and only if $T>0$ or the 
$(3,3)$-curvature matrix 
\begin{align*}
\begin{pmatrix}
R_{1212} & R_{1213} & R_{1223}\\
R_{1213} & R_{1313} & R_{1323}\\
R_{1223} & R_{1323} & R_{2323}
\end{pmatrix}
\end{align*}
has exactly one non-zero eigenvalue".

To only determine the solvability of the Gauss equation, we may apply Jacobowitz' result.
However for our purpose, we must know the explicit form of the solution of the Gauss equation, 
when we consider the derived Gauss equation.
In the following we reformulate Theorem~\ref{Thomas} and the result of Jacobowitz \cite{J1} to 
the form fitted to our purpose.

\subsection{Solvable cases}
In this subsection, we study the Gauss equation in codimension $1$ for three-dimensional 
solvable metric Lie algebras. 
Let $(\g, \naiseki)$ be a three-dimensional non-flat solvable metric Lie algebra.
By Milnor-type theorems (Proposition~\ref{M1}) and Lemma~\ref{R}, we can take an orthonormal 
basis $\{ x_1, x_2, x_3 \}$ of $(\g, k \naiseki)$ for some $k>0$ so that
$R_{1223}=R_{1323}=0$ and $R_{1212} \neq 0$. 
In fact, three-dimensional flat solvable metric Lie algebras are exhausted by $\mathbb{R}^3$ with any 
inner product $\naiseki$ and $(\mathfrak{r}^{\prime}_{3, \alpha}, k \naiseki)$ with $(\alpha, 
\lambda)=(0,1)$, following the notations in Proposition~\ref{M1} (3).
Then by Lemma~\ref{R}, we can easily verify the property $R_{1212} \neq 0$ for the remaining non-flat 
cases, including the Lie algebras $\g = \mathfrak{h}_3$ and $\mathfrak{r}_{3,1}$, where the inner 
products are uniquely determined up to isometry and scaling.

First of all we prove a lemma, which is a refinement of Theorem~\ref{Thomas}, fitted to our setting.

\begin{lem}\label{Lem G}
Let $(\g, \naiseki)$ be a three-dimensional metric Lie algebra.
Suppose that $R_{1223}=R_{1323} = 0$ and $R_{1212} \neq 0$. 
\begin{enumerate}
\item When $R_{2323} \neq 0$, $(\g,\naiseki)$ has a solution of the Gauss equation in codimension $1$ 
if and only if the following inequality holds:
\begin{align*}
S:=\frac{R_{1212} R_{1313} - R_{1213}^2}{R_{2323}}>0.
\end{align*}
Moreover the solution is unique up to sign, and is given by
\begin{align}\label{sol}
h_{11} = \pm \sqrt{S}, \
h_{22} = \frac{R_{1212}}{h_{11}} ,\
h_{33} = \frac{R_{1313}}{h_{11}} ,\
h_{23} = \frac{R_{1213}}{h_{11}} ,\
h_{12} = h_{13} = 0.
\end{align}
\item When $R_{2323} = 0$, 
$(\g, \naiseki)$ has a solution of the Gauss equation in codimension $1$ if and only if the following 
equality holds: 
\begin{align*}
R_{1212}R_{1313}=R_{1213}^2.
\end{align*}  
In this case solutions are continuously deformable.
\end{enumerate}
\end{lem}

\begin{proof}
Take any three-dimensional metric Lie algebra $(\g, \naiseki)$.
Assume that $R_{1223}=R_{1323} = 0$ and $R_{1212} \neq 0$.
We first prove (1). Assume that $R_{2323} \neq 0$ and $(\g,\naiseki)$ has a solution of 
the Gauss equation in codimension $1$. 
As we mentioned in the previous subsection, the following system of quadratic equations hold:
\begin{eqnarray}
\left\{
\begin{array}{l}\label{GGG}
h_{11}h_{22}-h_{12}^2=R_{1212}\neq 0, \\
h_{11}h_{33}-h_{13}^2=R_{1313}, \\
h_{22}h_{33}-h_{23}^2=R_{2323}\neq 0, \\
h_{11}h_{23}-h_{13}h_{12}=R_{1213}, \\
h_{12}h_{23}-h_{13}h_{22}=R_{1223}=0, \\
h_{12}h_{33}-h_{13}h_{23}=R_{1323}=0.
\end{array}
\right.
\end{eqnarray}
The last two equations can be expressed in the matrix form:
\begin{align*}
\begin{pmatrix} h_{23} & -h_{22} \\ h_{33} & -h_{23} \end{pmatrix}
\begin{pmatrix} h_{12} \\ h_{13} \end{pmatrix} = 
\begin{pmatrix} 0 \\ 0 \end{pmatrix}.
\end{align*}
Since the determinant of this square matrix is $h_{22}h_{33}-h_{23}^2 = R_{2323} \neq 0$, we have 
$h_{12} = h_{13} = 0$.
Then, substituting $h_{12}=0$ into the first equation in (\ref{GGG}), 
we have $h_{11}h_{22} = R_{1212} \neq 0$, from which we have $h_{11} \neq 0$.
This gives that
\begin{align*}
h_{22}=\frac{R_{1212}}{h_{11}}, \
h_{33}=\frac{R_{1313}}{h_{11}}, \
h_{23}=\frac{R_{1213}}{h_{11}}.
\end{align*}
By substituting these equalities into the remaining equation $h_{22}h_{33}-h_{23}^2=R_{2323}$, we obtain
\begin{align*}
S= \frac{R_{1212} R_{1313} - R_{1213}^2}{R_{2323}}=h_{11}^2 >0.
\end{align*}
Therefore $S>0$ is a necessary condition so that $(\g,\naiseki)$ has a solution of the Gauss equation 
in codimension $1$.
From the above calculations, the solution is uniquely determined up to sign, and is explicitly given by 
\begin{align*}
h_{11} = \pm \sqrt{S}, \
h_{22} = \frac{R_{1212}}{h_{11}} , \
h_{33} = \frac{R_{1313}}{h_{11}} , \
h_{23} = \frac{R_{1213}}{h_{11}} , \
h_{12}=h_{13}=0.
\end{align*}

\noindent
Conversely, if $S>0$, then one can easily show that the above $h_{ij}$ gives a solution 
of the Gauss equation in codimension $1$.

Next we show (2).
Assume that $R_{2323}=0$, and $(\g,\naiseki)$ has a solution of the Gauss equation in codimension $1$. 
Then the following system of quadratic equations hold:
\begin{align*}
\left\{
\begin{array}{l}
h_{11}h_{22}-h_{12}^2=R_{1212}\neq 0, \\
h_{11}h_{33}-h_{13}^2=R_{1313}, \\
h_{22}h_{33}-h_{23}^2=R_{2323}=0, \\
h_{11}h_{23}-h_{13}h_{12}=R_{1213}, \\
h_{12}h_{23}-h_{13}h_{22}=R_{1223}=0, \\
h_{12}h_{33}-h_{13}h_{23}=R_{1323}=0. 
\end{array}
\right.
\end{align*}
In case $h_{12} \neq 0$, we have 
\begin{align*}
h_{23}=\frac{h_{13}h_{22}}{h_{12}}, \quad h_{33}=\frac{h_{13} h_{23}}{h_{12}}=\frac{h_{13}^2 h_{22}}{h_{12}^2}
\end{align*}
from the last two equations $h_{12} h_{23}-h_{13}h_{22}=0$ and $h_{12} h_{33}-h_{13}h_{23}=0$.
Hence we obtain
\begin{align*}
R_{1213}=h_{11}h_{23}-h_{13}h_{12}=\frac{h_{11}h_{13}h_{22}-h_{13}h_{12}^2}{h_{12}}=\frac{h_{13}}{h_{12}} R_{1212}.  
\end{align*} 
Similarly, we obtain 
\begin{align*}
R_{1313}=h_{11}h_{33}-h_{13}^2=\frac{h_{13}^2}{h_{12}^2}R_{1212}.
\end{align*}
Hence one see that
\begin{align*}
R_{1212} R_{1313}=\frac{h_{13}^2}{h_{12}^2}R_{1212}^2=R_{1213}^2.  
\end{align*} 

In case $h_{12} = 0$, from the equations $h_{11}h_{22}-h_{12}^2 = R_{1212} \neq 0$ and 
$h_{12}h_{23}-h_{13}h_{22}=R_{1223}=0$, one can see that $h_{11}h_{22} \neq 0$ and $h_{13} =0$.
Then we have 
\begin{align*}
R_{1212} R_{1313} = h_{11}^2 h_{22} h_{33} = h_{11}^2 h_{23}^2=R_{1213}^2
\end{align*}
since $h_{12}=h_{13}=0$ and $h_{22}h_{33}=h_{23}^2$.

Conversely, when $R_{1212} R_{1313} =R_{1213}^2$ holds, we have solutions of the Gauss equation 
by taking arbitrary $h_{11} (\neq 0)$ and $h_{12}$, and putting remaining $h_{ij}$ as follows:
\begin{align*}
h_{22}= \frac{R_{1212}+h_{12}^2}{h_{11}}, \ h_{23}=\frac{R_{1213}}{R_{1212}}h_{22}, 
\ h_{33}=\frac{R_{1213}^2}{R_{1212}^2}h_{22}, \ h_{13}=\frac{R_{1213}}{R_{1212}}h_{12}.
\end{align*}
\end{proof}

We apply this lemma for each three-dimensional non-flat solvable metric Lie algebra.
Note that the solvability of the Gauss equation is unchanged if we replace $\naiseki$ with 
$k \naiseki$ for any $k>0$.

\subsubsection{The cases of $\mathfrak{h}_3$ and $\mathfrak{r}_{3,1}$}
We first consider the case of $\mathfrak{h}_3$.
Let $\{ e_1, e_2, e_3\}$ be the basis of $\mathfrak{h}_3$ so that the nonzero bracket relation is given by
$[e_1, e_2]=e_3 $, and $\naiseki_0$ be the inner product so that the basis $\{ e_1, e_2, e_3\}$ is orthonormal.
Recall that, the inner product on $\mathfrak{h}_3$ is unique up to isometry and scaling. 
\begin{prop}\label{g1}
$(\mathfrak{h}_3, \naiseki_0)$ does not have a solution of the Gauss equation in codimension $1$. 
\end{prop}
\begin{proof}
By Lemma~\ref{R}, the Riemannian curvature tensors are given by
\begin{align*}
& R_{1212}=-\frac{3}{4}, \quad R_{1313}=R_{2323}=\frac{1}{4}, \\
& R_{1213}=R_{1223}=R_{1323}=0, 
\end{align*} 
which satisfies the assumptions in Lemma~\ref{Lem G} (1).
A direct calculation shows that 
\begin{align*}
\frac{R_{1212} R_{1313} - R_{1213}^2}{R_{2323}}=-\frac{3}{4}<0.
\end{align*}
By Lemma~\ref{Lem G} (1), $(\mathfrak{h}_3, \naiseki_0)$ does not have a solution of the Gauss equation 
in codimension $1$.
\end{proof}

\begin{rem}
This fact has been already proved by Rivertz (\cite{R2}), and Masal'tsev (\cite{Mas1}), 
Borisenko (\cite{Bor}) have improved his result.
Namely, Masal'tsev proved that the $2n+1$-dimensional Heisenberg group equipped 
with any left-invariant Riemannian metric can not be locally isometrically embedded into the 
$2n+2$-dimensional Euclidean space, and Borisenko further improved Masal'tsev's 
result, by showing the non-existence of embeddings into the $4n$-dimensional Euclidean space. 
For details on the previous results of $\mathfrak{h}_3$, see \cite{Mas1}, \cite{Bor}.
Kaneda have deeply studied local isometric embeddings of the $3$-dimensional Heisenberg group 
in the case of codimension $2$ (\cite{K2}).
\end{rem}

Next we consider the Lie algebra $\mathfrak{r}_{3, 1}$.
In this case, inner products on $\mathfrak{r}_{3, 1}$ are also unique up to isometry and scaling. 
We take the inner product $\naiseki_0$ so that the basis $\{ e_1, e_2, e_3\}$ of  $\mathfrak{r}_{3, 1}$, 
whose nonzero bracket relations are given by 
$[e_1, e_2]=e_2, [e_1, e_3 ]=e_3 $, is orthonormal. 
By Lemma~\ref{Lem G} (1) and Lemma~\ref{R}, one can easily prove the following proposition.

\begin{prop}
$(\mathfrak{r}_{3, 1} ,\naiseki_0)$ does not have a solution of the Gauss equation in codimension $1$. 
\end{prop}

\noindent
The Lie group whose metric Lie algebra is $(\mathfrak{r}_{3, 1},\naiseki_0)$ is a space of 
constant negative curvature, and the above result is a well-known fact.
(It is known that it can be locally isometrically embedded into $\mathbb{R}^5$, and this gives 
the least dimensional local isometric embedding.)

\subsubsection{The case of $\g=\mathfrak{r}_3$}
We take any inner product $\naiseki$ on $\mathfrak{r}_{3}$. 
Then, by Proposition~\ref{M1} (1), there exist $\lambda > 0$, $k > 0$, 
and an orthonormal basis $\{x_1, x_2, x_3\}$ with respect to $k \naiseki$ 
such that the non-zero bracket relations are given by 
\begin{align*}
[x_1,x_2] = x_2 + 2\lambda x_3 , \quad [x_1,x_3] = x_3. 
\end{align*}
Moreover, in the following arguments, we may assume $k=1$ since the existence or non-existence of 
local isometric embeddings does not depend on the choice of the scaling.
Write the above $\naiseki$ as $\naiseki_{\lambda}$ in the following.
It is sufficient to investigate the solvability of the Gauss equation of metric Lie algebras 
$(\mathfrak{r}_{3}, \naiseki_{\lambda})$.
\begin{prop}\label{G1}
The metric Lie algebra $(\mathfrak{r}_{3},\naiseki_{\lambda})$ has a solution of the Gauss equation 
in codimension $1$ if and only if  $\frac{1}{\sqrt{3}} < \lambda < 1$.
\end{prop}
\begin{proof}
By Lemma~\ref{R}, the Riemannian curvature tensors are given by
\begin{align*}
&R_{1212}=-(3 \lambda^2+1), \\
&R_{1313}=R_{2323}=\lambda^2-1, \\
&R_{1213}=-2 \lambda, \\
&R_{1223}=R_{1323}=0. 
\end{align*} 
We see that $R_{2323}=-1+\lambda^2 =0$ if and only if $\lambda =1$ since $\lambda>0$.
In the case $\lambda \neq 1$, 
we can apply Lemma~\ref{Lem G} (1).
A direct calculation shows that
\begin{align*}
S=\frac{R_{1212} R_{1313} - R_{1213}^2}{R_{2323}} 
&=\frac{-(3 \lambda^2+1)(\lambda^2-1)-(-2 \lambda)^2}{\lambda^2-1}\\
&=\frac{-3 \lambda^4-2\lambda^2+1}{\lambda^2-1}\\
&=\frac{-(3\lambda^2-1)(\lambda^2+1)}{\lambda^2-1}.
\end{align*}  
It is easily see that  
$S >0$
if and only if $\frac{1}{\sqrt{3}} < \lambda < 1$. 

In the case $\lambda =1$,
we have $R_{1212}=-4\neq 0, R_{1313}=R_{2323}=R_{1223}=R_{1323}=0$ and $R_{1213}=-2$.
Then one can see that
\begin{align*}
R_{1212}R_{1313}=0 \neq 4 = R_{1213}^2.
\end{align*}
This shows that $(\mathfrak{r}_{3}, \naiseki_1)$ dose not have a solution of the Gauss 
equation in codimension $1$ by Lemma~\ref{Lem G} (2).
\end{proof}

We can treat the remaining cases in a similar way, although the results are a little 
complicated.

\subsubsection{The case of $\g=\mathfrak{r}_{3, \alpha} \ (-1 \leq \alpha <1)$}
We take any inner product $\naiseki$ on $\mathfrak{r}_{3, \alpha}$. 
By Proposition~\ref{M1} (2) there exist
$\lambda \in \mathbb{R}$, $k>0$ and an orthonormal basis $\{x_1, x_2, x_3\}$ with respect to $k \naiseki$ 
such that the non-zero bracket relations are given by 
\begin{align*}
[x_1,x_2]=x_2+2\lambda (\alpha-1)x_3, \quad [x_1,x_3]=\alpha x_3.
\end{align*}
As before, we may set $k=1$.
Let us denote by $\naiseki_{\lambda}$ the above inner product $\naiseki$.

\begin{prop}\label{G2} 
Let $\alpha \in [-1,1)$.
For the metric Lie algebra $(\mathfrak{r}_{3, \alpha}, \naiseki_{\lambda})$, we have the following: 
\begin{enumerate}
\item When $\alpha = 0$, 
$(\mathfrak{r}_{3, 0},\naiseki_{\lambda})$ has a solution of the Gauss equation in codimension $1$ if and only if $\lambda=0$. 
\item When $\alpha \in [-1, 0)$, 
$(\mathfrak{r}_{3, \alpha},\naiseki_{\lambda})$ has a solution of the Gauss equation in codimension $1$ if and only if
\begin{align*}
|\lambda| < \frac{\sqrt{\sqrt{(1+\alpha^2)^2+12 \alpha^2}-(1+\alpha^2)}}{\sqrt{6}(1-\alpha)}.
\end{align*}
\item When $\alpha \in (0, 1)$, 
$(\mathfrak{r}_{3, \alpha},\naiseki_{\lambda})$ has a solution of the Gauss equation in codimension $1$ if and only if
\begin{align}\label{ineq_1}
\frac{\sqrt{\sqrt{(1+\alpha^2)^2+12 \alpha^2}-(1+\alpha^2)}}{\sqrt{6}(1-\alpha)}<|\lambda| < \frac{\sqrt{\alpha}}{1-\alpha}.
\end{align}
\end{enumerate}
\end{prop}

\begin{proof}
For each $\alpha \in [-1, 1)$, we take the inner product $\naiseki_{\lambda}$ on $\mathfrak{r}_{3, \alpha}$. 
By Lemma~\ref{R}, the Riemannian curvature tensors are given by
\begin{align*}
&R_{1212}=-1-3\lambda^2 (\alpha -1)^2 \neq 0, \\
&R_{1313}=-\alpha^2 +\lambda^2 (\alpha-1)^2, \\
&R_{2323}=-\alpha+\lambda^2 (\alpha-1)^2, \\
&R_{1213}=-2\lambda \alpha (\alpha -1), \\
&R_{1223}=R_{1323}=0. 
\end{align*}  

First, we assume $\alpha=0$.
In this case the condition $R_{2323}\neq 0$ is equivalent to $\lambda \neq 0$. 
In case $\lambda \neq 0$, we can apply Lemma~\ref{Lem G} (1). 
A direct calculation show that
\begin{align*}
S = -1-3\lambda^2<0,
\end{align*} 
which concludes that $(\mathfrak{r}_{3, 0}, \naiseki_{\lambda})$ does not admit a solution of 
the Gauss equation in codimension $1$ when $\lambda \neq 0$.
In case $\lambda = 0$, by applying Lemma~\ref{Lem G} (2), we obtain that
\begin{align*}
R_{1212}R_{1313}-R_{1213}^2 = 0,
\end{align*} 
which concludes that $(\mathfrak{r}_{3, 0}, \naiseki_0)$ has solutions of the Gauss equation in codimension $1$.
We complete the proof of the case (1).

Next we consider the case of $-1 \leq \alpha < 0$. 
Then we see that 
$R_{2323}>0$. 
Hence we can apply Lemma~\ref{Lem G} (1), and $S>0$ if and only if $R_{1212} R_{1313} -R_{1213}^2>0$.
By direct calculations, we obtain that 
\begin{align*}
R_{1212} R_{1313} -R_{1213}^2 = -\{ 3(\alpha-1)^4 \lambda^4+(\alpha-1)^2(\alpha^2+1)\lambda^2-\alpha^2 \}
\end{align*}
and it follows from this equation we see that $R_{1212} R_{1313} -R_{1213}^2 >0$ if and only if 
\begin{align*}
|\lambda| < \frac{\sqrt{\sqrt{(1+\alpha^2)^2+12\alpha^2}-(\alpha^2+1)}}{\sqrt{6}(1-\alpha)},
\end{align*} 
which shows (2).

Finally we consider the case of $0<\alpha <1$. When $R_{2323}=0$, we have that $\lambda^2=\frac{\alpha}{(\alpha-1)^2}$ and 
\begin{align*}
R_{1212} R_{1313}-R_{1213}^2=-\alpha (\alpha+1)^2\neq 0.
\end{align*} 
By Lemma~\ref{Lem G} (2), $(\mathfrak{r}_{3, \alpha}, \naiseki_{\lambda})$ dose not admit 
a solution of the Gauss equation in codimension $1$.
Next we consider the case $R_{2323} \neq 0$.
In this case, the inequality $S>0$ holds if and only if 
\begin{align*}
\{ 3(\alpha-1)^4 \lambda^4+(\alpha-1)^2(\alpha^2+1)\lambda^2-\alpha^2 \}
\{ (\alpha-1)^2\lambda^2-\alpha \}<0.
\end{align*}
Since the inequality
\begin{align*}
0<\frac{\sqrt{(1+\alpha^2)^2+12 \alpha^2}-(1+\alpha^2)}{6(1-\alpha)^2}<\frac{\alpha}{(1-\alpha)^2}
\end{align*}
holds for $0<\alpha <1$, the condition $S>0$ is equivalent to 
\begin{align*}
\frac{\sqrt{\sqrt{(1+\alpha^2)^2+12 \alpha^2}-(1+\alpha^2)}}{\sqrt{6}(1-\alpha)}<|\lambda| 
< \frac{\sqrt{\alpha}}{1-\alpha}.
\end{align*}
Therefore, by Lemma ~\ref{Lem G} (1), $(\mathfrak{r}_{3, \alpha}, \naiseki_{\lambda})$ admits 
a solution of the Gauss equation in codimension $1$ if and only if the inequality~(\ref{ineq_1}) holds.
We complete the proof.
\end{proof}

In Figure 1 we draw the range of $(\alpha, \lambda)$ where the metric Lie algebra 
$(\mathfrak{r}_{3, \alpha},\naiseki_{\lambda})$ admits a solution of the Gauss equation in codimension $1$.

\begin{center}
\includegraphics*[bb=235 60 510 330,clip,width=7cm]{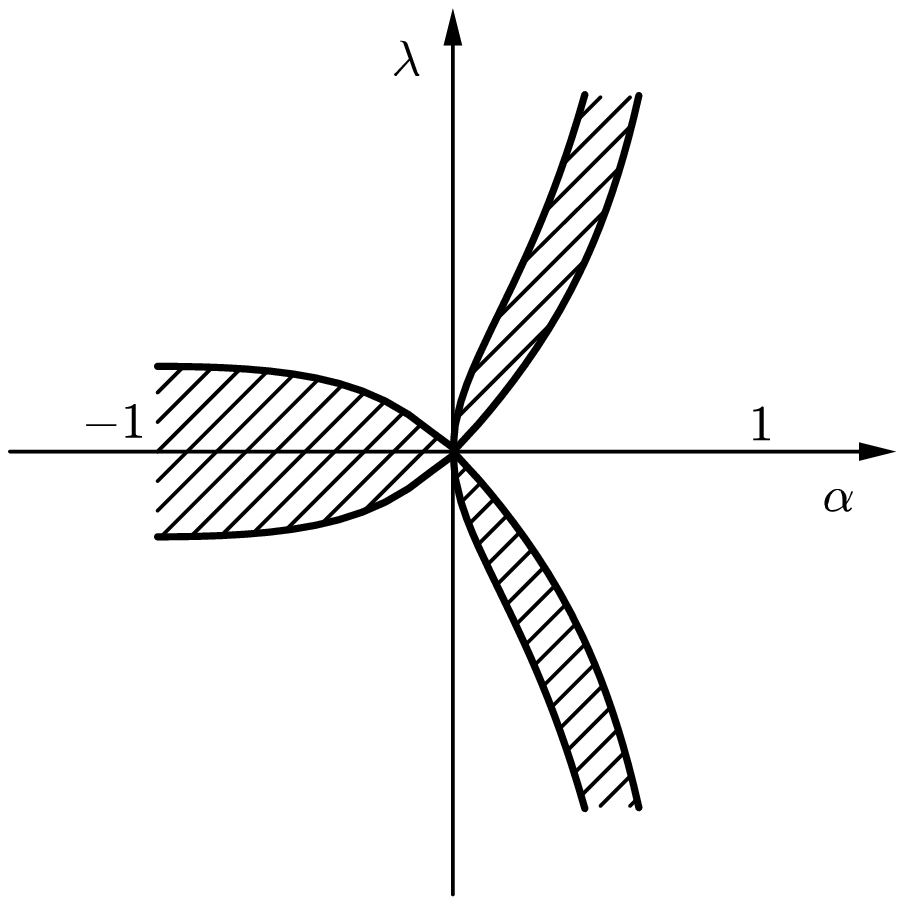}

Figure 1
\end{center}

\subsubsection{The case of $\g=\mathfrak{r}^{\prime}_{3, \alpha} \ (\alpha \geq 0)$}
We take any inner product $\naiseki$ on $\mathfrak{r}^{\prime}_{3, \alpha}$. 
By Proposition~\ref{M1} (3),
there exist
$\lambda \geq 1$, $k>0$ and an orthonormal basis $\{x_1, x_2, x_3\}$ with respect to $k \naiseki$ 
such that the non-zero bracket relations are given by 
\begin{align*}
[x_1, x_2]=\alpha x_2 - \lambda x_3, \quad [x_1, x_3]=\frac{1}{\lambda} x_2 + \alpha x_3.
\end{align*} 
We set $k=1$, and 
let us denote by $\naiseki_{\lambda}$ the above inner product $\naiseki$.

\begin{prop}\label{G3}
Let $\alpha \geq 0$ and $\lambda \geq 1$.
For the metric Lie algebra $(\mathfrak{r}^{\prime}_{3, \alpha}, \naiseki_{\lambda})$, we have the following:
\begin{enumerate}
\item When $\alpha=0$, 
$(\mathfrak{r}^{\prime}_{3, 0},\naiseki_{\lambda})$ has a solution of the Gauss equation 
in codimension $1$ if and only if $\lambda=1$.
\item When $\alpha >0$, 
$(\mathfrak{r}^{\prime}_{3, \alpha},\naiseki_{\lambda})$ has a solution of the Gauss equation 
in codimension $1$ if and only if 
\begin{align*}
\sqrt{\frac{1+2\alpha^2 +2 \sqrt{1+\alpha^2+\alpha^4}}{3}} < \lambda <  \alpha +\sqrt{1+\alpha^2}.
\end{align*}
\end{enumerate}
\end{prop}

\begin{proof}
For each $\alpha \geq 0$ and $\lambda \geq 1$, 
we take the inner product $\naiseki_{\lambda}$ on $\mathfrak{r}^{\prime}_{3, \alpha}$.
By Lemma~\ref{R}, the Riemannian curvature tensors are given by
 \begin{align*}
& R_{1212}=-\alpha^2-\frac{1}{4} \left( \lambda -\frac{1}{\lambda} \right) \left( 3\lambda+\frac{1}{\lambda} \right), \\
& R_{1313}=-\alpha^2+\frac{1}{4} \left( \lambda -\frac{1}{\lambda} \right) \left( \lambda +\frac{3}{\lambda} \right), \\
& R_{2323}=-\alpha^2+\frac{1}{4} \left( \lambda -\frac{1}{\lambda} \right)^2, \\
& R_{1213}=\alpha \left( \lambda -\frac{1}{\lambda } \right), \\
& R_{1223}=R_{1323}=0. 
 \end{align*}
In the case of $\alpha =0$, the metric Lie algebra $(\mathfrak{r}^{\prime}_{3, 0},\naiseki_{\lambda})$ 
is flat if and only if $\lambda=1$, and in this case the Gauss equation clearly has a solution in codimension $1$.
Assume $\alpha=0$ and $\lambda > 1$.
Then we have $R_{1212} \neq 0$, $R_{2323} > 0$ and 
\begin{align*}
R_{1212}R_{1313}-R_{1213}^2 = -\frac{1}{16} \left( \lambda-\frac{1}{\lambda} \right)^2 
\left( 3 \lambda +\frac{1}{\lambda} \right) \left( \lambda +\frac{3}{\lambda} \right) <0.
\end{align*}
Hence we have $S<0$, and by Lemma~\ref{Lem G} (1), the metric Lie algebra $(\mathfrak{r}^{\prime}_{3, 0}, 
\naiseki_{\lambda})$ 
does not admit a solution of the Gauss equation in codimension $1$ .
We conclude the case (1).

Next we consider the case $\alpha >0$. 
By direct calculation, we obtain that 
\begin{align*}
& R_{1212}R_{1313}-R_{1213}^2 \\
& \qquad =-\frac{1}{16 \lambda^4} \left\{ 3 \lambda^4 -2 (1+2\alpha^2)\lambda^2-1 \right\} 
\left\{ \lambda^4+2 (1+2\alpha^2) \lambda^2-3  \right\}.
\end{align*}
Then, considering the equality $R_{1212}R_{1313}-R_{1213}^2=0$ as a quadratic equation on $\lambda^2$, 
we have 
\begin{align*}
\lambda^2=\frac{1}{3} \left( 1+2\alpha^2\pm 2 \sqrt{1+\alpha^2+\alpha^4} \right), \quad 
-(1+2\alpha^2)\pm 2 \sqrt{1+\alpha^2+\alpha^4}.
\end{align*}
Similarly, from the equation $R_{2323}=\{ \lambda^4-2(1+2\alpha^2)\lambda^2+1 \}/(4\lambda^2)=0$, we have 
$$
\lambda^2= 1+2\alpha^2 \pm 2 \alpha \sqrt{1+\alpha^2} \: \left( = (\alpha \pm \sqrt{1+\alpha^2}\,)^2 \right).
$$
Using the property $\alpha>0$, we can easily show the following inequalities: 
\begin{equation}
\begin{split}\label{ineq_2}
& -(1+2 \alpha^2)-2 \sqrt{1+\alpha^2+\alpha^4} 
< \frac{1}{3} \left( 1+2 \alpha^2-2 \sqrt{1+\alpha^2+\alpha^4} \right) < 0 \\
& \quad < 1+2 \alpha^2-2 \alpha \sqrt{1+\alpha^2} 
< -(1+2 \alpha^2)+2 \sqrt{1+\alpha^2+\alpha^4} < 1 \\
& \quad \quad < \frac{1}{3} \left( 1+2 \alpha^2+2 \sqrt{1+\alpha^2+\alpha^4} \right) 
< 1+2\alpha^2+2\alpha \sqrt{1+\alpha^2}.
\end{split}
\end{equation}
Therefore in case $R_{2323} \neq 0$, combining the condition $\lambda^2 \geq 1$, it follows that 
the inequality $S>0$ holds if and only if 
$$
\frac{1}{3} \left( 1+2 \alpha^2+2 \sqrt{1+\alpha^2+\alpha^4} \right) < \lambda^2  
< 1+2\alpha^2+2\alpha \sqrt{1+\alpha^2}.
$$
(In this case, we have $R_{1212}R_{1313}-R_{1213}^2<0$ and $R_{2323}<0$.)
Since $\alpha>0$ and $\lambda \geq 1$, the above condition is equivalent to
\begin{align*}
\sqrt{\frac{1+2\alpha^2 +2 \sqrt{1+\alpha^2+\alpha^4}}{3}} < \lambda < \alpha +\sqrt{1+\alpha^2}. 
\end{align*}

In case $R_{2323}=0$, we have $\lambda^2= 1+2\alpha^2 +2 \alpha \sqrt{1+\alpha^2}$ since 
$\lambda^2 \geq 1$.
In this case $R_{1212}R_{1313}-R_{1213}^2 \neq 0$ from the above inequalities~(\ref{ineq_2}). 
Hence by Lemma~\ref{Lem G} (2), the metric Lie algebra $(\mathfrak{r}^{\prime}_{3, \alpha}, 
\naiseki_{\lambda})$ does not admit a solution of the Gauss equation in codimension $1$.
\end{proof}

In Figure 2 we draw the range of $(\alpha, \lambda)$ where the metric Lie algebra 
$(\mathfrak{r}^{\prime}_{3, \alpha}, \naiseki_{\lambda})$ admits a solution of the 
Gauss equation in codimension $1$.

\begin{center}
\includegraphics*[bb=190 130 460 345,clip,width=7cm]{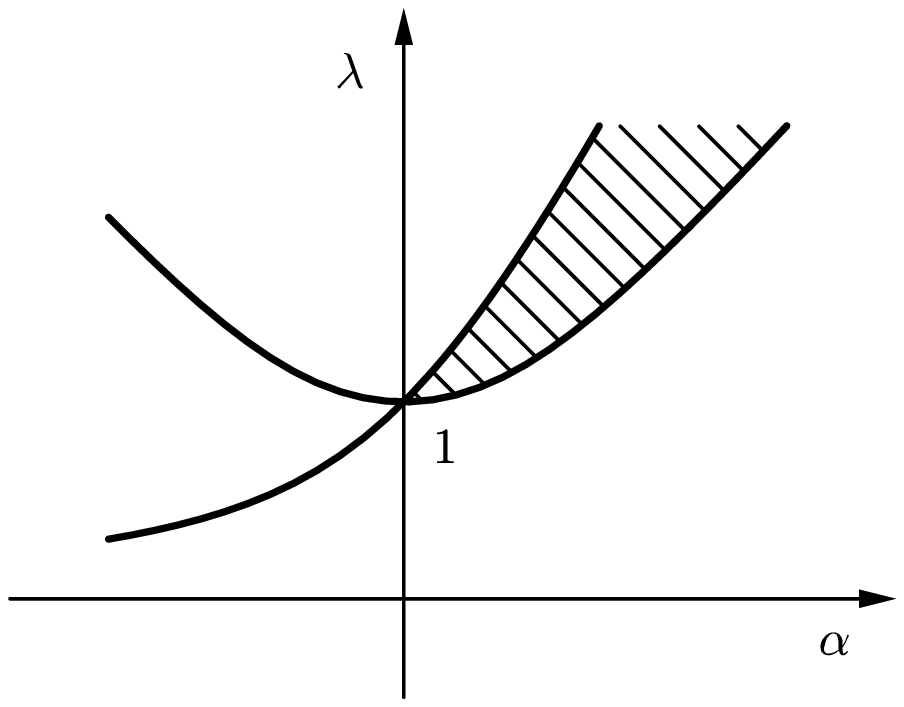}

Figure 2 \;\;\;\;
\end{center}

\subsection{Simple cases}
In this subsection, we study solutions of the Gauss equations of three-dimensional simple metric Lie algebras in codimension $1$. 

From Proposition~\ref{Ms}, for each three-dimensional simple metric Lie algebra $(\g, \naiseki)$, 
there exist $\lambda_2 >0$, $\lambda_3 \neq 0$, $k>0$ and an orthonormal basis $\{ x_1, x_2, x_3 \}$ 
of $\g$ with respect to $k \naiseki$ such that non-zero bracket relations are given by 
\begin{align*}
[x_1, x_2]= \lambda_3 x_3, \ [x_3, x_1]=\lambda_2 x_2, \ [x_2, x_3]=x_1.
\end{align*} 
We can assume that $k=1$ as before.
We here put 
$$
\lambda_2:=\frac{u+v}{2}, \qquad
\lambda_3:=\frac{-u+v}{2}.
$$ 
Then we have 
\begin{align*}
[x_1, x_2]= \frac{-u+v}{2} x_3, \ [x_3, x_1]=\frac{u+v}{2} x_2, \ [x_2, x_3]=x_1,
\end{align*}
and since $\lambda_2>0$ and $\lambda_3 \neq 0$, we have $u \neq \pm v$.
Let us denote by $\naiseki_{(u,v)}$ the inner product $\naiseki$.

\begin{prop}\label{G_simple}
Let $(\g, \naiseki_{(u,v)})$ be the three-dimensional simple metric Lie algebra. 
Then, $(\g, \naiseki_{(u,v)})$ admits a solution of the Gauss equation in codimension $1$ if and only if
\begin{enumerate}
\item \rule{0cm}{0.5cm} When $u=0$, 

\quad \rule{0cm}{0.5cm} $\displaystyle 2(v-1)>1$.
\item \rule{0cm}{0.8cm} When $0 < |u| \leq 1$, 

\quad $1-u^2< 2(v-1) < \displaystyle{\left| \frac{1-u^2}{u} \right|}$ \;\; or \;\;
\rule{0cm}{0.8cm} $2(v-1)< \displaystyle{- \left| \frac{1-u^2}{u} \right|}$.
\item \rule{0cm}{0.5cm} When $|u| > 1$, 

\quad $|2(v-1)| < \displaystyle{\left| \frac{1-u^2}{u} \right|}$ \;\; or \;\;
\rule{0cm}{0.8cm} $2(v-1)< 1-u^2$.
\end{enumerate}
\end{prop}

\begin{proof}
By Lemma~\ref{R}, the Riemannian curvature tensors are given by the following: 
\begin{align*}
& R_{1212} = \frac{1}{4}\{ 1-u^2+2(v-1)u \}, \\
& R_{1313} = \frac{1}{4}\{ 1-u^2-2(v-1)u \}, \\ 
& R_{2323} = -\frac{1}{4}\{ 1-u^2-2(v-1) \}, \\
& R_{1213} = R_{1223} = R_{1323} = 0.
\end{align*}
Then we can easily see that the following four conditions are equivalent:

(i) \rule{0cm}{0.5cm} $R_{1212} = R_{1313} = 0$,

(ii) \rule{0cm}{0.5cm} $R_{1212} = R_{2323} = 0$,

(iii) \rule{0cm}{0.5cm} $R_{1313} = R_{2323} = 0$,

(iv) \rule[-0.3cm]{0cm}{0.8cm} $u^2=1$ and $v=1$.

\noindent
Since $u \neq \pm v$, it follows that none of the above four conditions actually occurs.

Now assume that $(\g, \naiseki_{(u,v)})$ admits a solutions of the Gauss equation in codimension $1$.
It is expressed in the following form:  
\begin{numcases}{}
\label{n1} h_{11}h_{22}-h_{12}^2=R_{1212}, \\
\label{n2} h_{11}h_{33}-h_{13}^2=R_{1313}, \\
\label{n3} h_{22}h_{33}-h_{23}^2=R_{2323}, \\
\label{n4} h_{11}h_{23}-h_{13}h_{12}=0, \\
\label{n5} h_{12}h_{23}-h_{13}h_{22}=0, \\
\label{n6} h_{12}h_{33}-h_{13}h_{23}=0.
\end{numcases}
We first show that $h_{11} \neq 0$.
Assume $h_{11}=0$.
If $h_{12} \neq 0$, then from (\ref{n4}) we have $h_{13}=0$.
Then from (\ref{n5}) and (\ref{n6}), we have $h_{23}=h_{33}=0$, and hence we obtain 
$R_{1313} = R_{2323} = 0$ from (\ref{n2}) and (\ref{n3}).
This is a contradiction.
Hence we have $h_{12}=0$.
In the same way we can show the equality $h_{13} = 0$.
Then since $h_{11}=h_{12}=h_{13}=0$, we have $R_{1212}=R_{1313}=0$ from (\ref{n1}) and (\ref{n2}).
This is a contradiction.
Therefore we have $h_{11} \neq 0$.

Next, from the above equations (\ref{n1})$\sim$(\ref{n6}), we have 
\begin{align*}
R_{1212}R_{1313}&=h_{11}^2h_{22}h_{33}-h_{11}h_{33}h_{12}^2-h_{11}h_{22}h_{13}^2+h_{12}^2h_{13}^2\\
&=h_{11}^2h_{22}h_{33}-h_{11}h_{12}h_{13}h_{23}-h_{11}h_{12}h_{13}h_{23}+h_{11}^2 h_{23}^2\\
&=h_{11}^2h_{22}h_{33}-2h_{11}h_{23}h_{11}h_{23}+h_{11}^2h_{23}^2 \\
&=h_{11}^2h_{22}h_{33}-h_{11}^2h_{23}^2\\
&=h_{11}^2R_{2323}.
\end{align*}
If $R_{2323}=0$, then we have $R_{1212}R_{1313}=0$.
Then it follows that $R_{1212} = R_{2323}=0$ or $R_{1313}=R_{2323}=0$, which contradicts to the 
assumption $u \neq \pm v$.
Hence we have $R_{2323}\neq0$.
In particular we have 
$R_{1212}R_{1313}R_{2323} = h_{11}^2R_{2323}^2 > 0$.
Conversely if $R_{1212}R_{1313}R_{2323} > 0$, then we have 
\begin{align*}
T=\begin{vmatrix}
R_{1212} & R_{1213} & R_{1223}\\
R_{1213} & R_{1313} & R_{1323}\\
R_{1223} & R_{1323} & R_{2323}
\end{vmatrix} = R_{1212}R_{1313}R_{2323} > 0.
\end{align*}
From the converse part of Theorem ~\ref{Thomas}, it follows that the Gauss equation has a 
solution in codimension $1$.

Our remaining problem is to solve the inequality $R_{1212}R_{1313}R_{2323} > 0$.
When $u=0$, it is easy to check that $R_{1212}R_{1313}R_{2323}>0$ if and only if 
$2(v-1)>1$.
When $u \neq 0$, the inequality $R_{1212}R_{1313}R_{2323}>0$ holds if and only if 
\begin{align*}
\left(2(v-1)+\frac{1-u^2}{u} \right)
\left(2(v-1)-\frac{1-u^2}{u} \right)
\left(2(v-1)-(1-u^2) \right) <0.
\end{align*}
Then we can show (2) and (3) by elementary arguments.
\end{proof}

We put $w=2(v-1)$, and in Figure 3 we draw the range of $(u, w)$ where the simple metric Lie algebra 
$(\g, \naiseki_{(u,v)})$ admits a solution of the Gauss equation in codimension $1$.

\begin{center}
\includegraphics*[bb=140 50 355 260,clip,width=7cm]{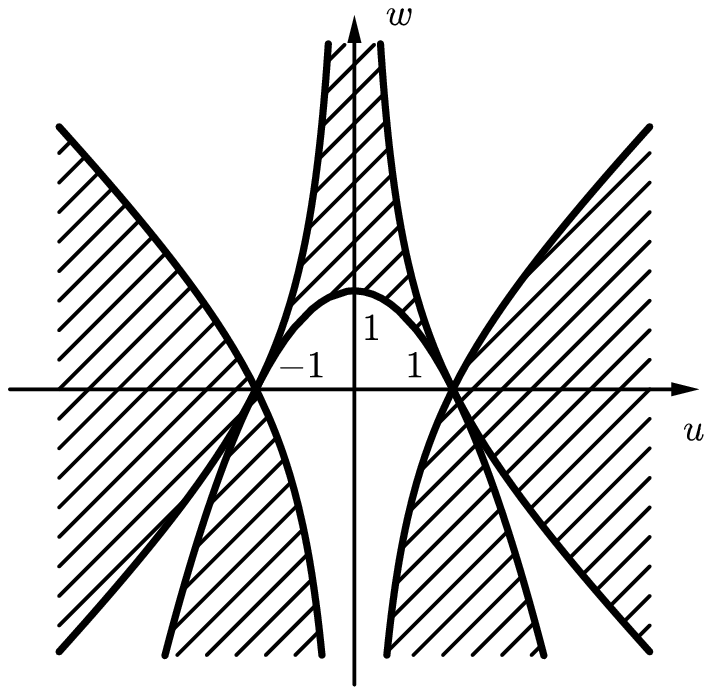}

Figure 3
\end{center}

\noindent
Note that the condition $u = \pm v$ is equivalent to $w=\pm 2u-2$, which are the tangent lines 
of the curves $w = \mp (1-u^2)/u$ at $(\pm 1,0)$, and we must exclude these two lines from 
the region in Figure 3, since the corresponding Lie algebras are not simple.
For example, the point $(u,w)=(1,-4)$, corresponding to the case $(u,v) = (1,-1)$, lies on the line $w=-2u-2$, 
and it is easy to see that it is isomorphic to the solvable metric Lie algebra $(\mathfrak{r}_{3, -1}, 
\naiseki_0)$ up to isometry and scaling.
Hence we must exclude it in our argument, though this metric Lie algebra actually admits a solution 
of the Gauss equation in codimension $1$, as we showed in Proposition ~\ref{G2} (2).

\section{The derived Gauss equation}
As stated in Introduction, it is well-known that the following metric Lie algebras $(\mathbb{R}^3, \naiseki)$, 
$(\mathfrak{r}_{3,0}, \naiseki_0)$, $(\mathfrak{r}^{\prime}_{3, 0}, \naiseki_1)$, and $(\mathfrak{so}(3), 
\naiseki_{(0,2)})$ can be locally isometrically embedded into $\mathbb{R}^4$, where notations are the same 
as in the previous sections.
In this section, we study solutions of the derived Gauss equation, 
and show that the remaining cases can not be locally isometrically embedded into $\mathbb{R}^4$.

Let $(\g, \naiseki)$ be a three-dimensional metric Lie algebra, 
and let $\{ x_1, x_2,x_3 \}$ be an orthonormal basis of $(\g, \naiseki)$.
We assume that there exists a local isometric embedding $F$ of
$(\g, \naiseki)$ into $\mathbb{R}^4$.
Then by Theorem~\ref{TK}, $F$ satisfies the derived Gauss equation: 
\begin{equation} 
\begin{split} \label{1}
& (\nabla R)(x_p,x_i,x_j,x_k,x_l) \\ 
& \qquad = \langle \nabla^3_{(x_p,x_i,x_k)} F, \nabla^2_{(x_j,x_l)} F \rangle_{\mathbb{R}^4} 
+ \langle \nabla^2_{(x_i,x_k)} F, \nabla^3_{(x_p,x_j,x_l)} F \rangle_{\mathbb{R}^4} \\
& \qquad \qquad - \langle \nabla^3_{(x_p,x_i,x_l)} F, \nabla^2_{(x_j,x_k)} F \rangle_{\mathbb{R}^4}
- \langle \nabla^2_{(x_i,x_l)} F, \nabla^3_{(x_p,x_j,x_k)} F \rangle_{\mathbb{R}^4}
\end{split}
\end{equation}
for each $i,j,k,l,p \in \{ 1,2,3 \}$.
We put $h_{ij}:=\langle (\nabla^2_{(x_i,x_j)}F)_e, N \rangle$ as before, where $e$ is the unit element 
of the Lie group $G$ and $N$ is a unit normal vector of the embedding $F$ at $e$.
Remind that $h_{ij}=h_{ji}$ and it satisfies the Gauss equation.
We here put $h_{ijk}:=\langle (\nabla^3_{(x_i,x_j,x_k)}F)_e, N \rangle$.
Then, rewriting the equation (\ref{1}), 
we have
\begin{align} \label{9}
& \nabla_{p} R_{ijkl} = h_{pik} h_{jl} 
+  h_{ik} h_{pjl} - h_{pil} h_{jk} - h_{il} h_{pjk} &
\end{align}
for each $i,j,k,l,p \in \{ 1,2,3 \}$, because the tangential part of $(\nabla^2_{(x_i,x_j)}F)_e$ is $0$.

From the integrability conditions 
\begin{align*}
\nabla_{(x_i, x_j, x_k)}^3 F & = \nabla_{(x_i, x_k, x_j)}^3 F, \\
\nabla_{(x_i, x_j, x_k)}^3 F & = \nabla_{(x_j, x_i, x_k)}^3 F - \nabla_{R(x_i, x_j)x_k}^1 F, 
\end{align*}
we have
$$
h_{ijk} = h_{ikj}
$$
and 
\begin{align*}
h_{ijk} &= \langle (\nabla_{(x_i, x_j, x_k)}^3 F)_e, N \rangle \\
& = \langle (\nabla_{(x_j, x_i, x_k)}^3 F)_e - (\nabla_{R(x_i, x_j)x_k}^1 F)_e, N \rangle \\
& = \langle (\nabla_{(x_j, x_i, x_k)}^3 F)_e , N \rangle \\
& = h_{jik}.
\end{align*}
Hence $h_{ijk}$ is a symmetric $3$-tensor.

In the following we consider (\ref{9}) as abstract linear equations on $h_{ijk}$ 
for a given derived curvature tensor $\nabla R$, and a given solution $h_{ij}$ of the Gauss 
equation, not assuming the existence of local isometric embeddings $F$ in advance.
Then (\ref{9}) serves as an obstruction to the existence of $F$.
Namely, if we can show the non-existence of a solution of the equation (\ref{9}) for a given 
$\nabla R$ and $h_{ij}$, we can conclude that $(\g, \naiseki)$ does not admit a local 
isometric embedding into $\mathbb{R}^4$, possessing $h_{ij}$ as its second fundamental form.
(In case $(M,g)$ is locally symmetric, i.e., the case of $\nabla R = 0$, the derived Gauss 
equation (\ref{9}) does not serve as an obstruction, because $h_{ijk} = 0$ gives a trivial 
solution of (\ref{9}) for any $h_{ij}$.)

Hereinafter, we only study three-dimensional metric Lie algebras $(\g, \naiseki)$ that admit 
a solution of the Gauss equation in codimension $1$, 
excluding the cases where the existence of local isometric embeddings 
into $\mathbb{R}^4$ is already known.
Explicitly, we only consider the following remaining cases:

\bigskip

$\bullet$
$(\mathfrak{r}_3, \naiseki_{\lambda})$ with $\displaystyle{\frac{1}{\sqrt{3}} < \lambda < 1}$.

$\bullet$
$(\mathfrak{r}_{3,\alpha}, \naiseki_{\lambda})$ with 
\begin{align*}
& |\lambda| < \frac{\sqrt{\sqrt{(1+\alpha^2)^2+12 \alpha^2}-(1+\alpha^2)}}{\sqrt{6}(1-\alpha)} 
& & \mbox{when} \quad \alpha \in [-1, 0), \\
& \rule{0cm}{1.1cm} \frac{\sqrt{\sqrt{(1+\alpha^2)^2+12 \alpha^2}-(1+\alpha^2)}}{\sqrt{6}(1-\alpha)}<|\lambda| 
< \frac{\sqrt{\alpha}}{1-\alpha} & & \mbox{when} \quad  \alpha \in (0, 1).
\end{align*} 

$\bullet$
$(\mathfrak{r}^{\prime}_{3,\alpha}, \naiseki_{\lambda})$ with $\alpha >0$ and 
\begin{align*}
& \sqrt{\frac{1+2\alpha^2 +2 \sqrt{1+\alpha^2+\alpha^4}}{3}} < \lambda < \alpha +\sqrt{1+\alpha^2}. 
\end{align*}

$\bullet$
Three-dimensional simple metric Lie algebra $(\mathfrak{g}, \naiseki_{(u,v)})$, where $(u,v)$ satisfies 
one of the conditions in Proposition~\ref{G_simple}, excluding the case $(\mathfrak{so}(3), \naiseki_{(0,2)})$.

\bigskip

\noindent
For the first three solvable cases we have $R_{1212} \neq 0$, $R_{2323} \neq 0$ and $R_{1223} = R_{1323} = 0$ 
from the arguments stated in the proof of Propositions~\ref{G1}, \ref{G2}, \ref{G3}.
Then by Lemma~\ref{Lem G} (1) we see that the solution of the Gauss equation is uniquely determined up to 
sign, and it is given by (\ref{sol}) for these solvable cases.
Note that the solvability of the derived Gauss equation (\ref{9}) does not depend on the 
choice of the sign of $h_{ij}$.

\subsection{Solvable cases}
First of all, we prove the following lemma. 
\begin{lem}\label{dGeq}
Let $(\g, \naiseki)$ be a three-dimensional solvable metric Lie algebra, 
and $\{ x_1,x_2,x_3 \}$ be the orthonormal basis whose non-zero bracket relations are given by 
\begin{align*}
[x_1, x_2] = a x_2 + 2b x_3, \quad [x_1, x_3] = 2c x_2 + d x_3 \quad (a,b,c,d \in \mathbb{R}). 
\end{align*}
Assume $R_{1212} \neq 0$, $R_{2323} \neq 0$, $R_{1223} = R_{1323} = 0$ and $(\g, \naiseki)$ admits 
a solution of the Gauss equation.
If $(\g,\naiseki)$ admits a solution of the derived Gauss equation, 
then the equality  
\begin{align*}
(b+c)(R_{1313}-R_{1212})+(a-d)R_{1213}=0
\end{align*}
holds.
\end{lem}

\begin{proof}
By the assumption that $(\g,\naiseki)$ has a solution of the derived Gauss equation, 
we obtain  
\begin{align}
\label{dG_4_1} \nabla_{2} R_{1223} & = h_{212} h_{23} + h_{12} h_{223} - h_{213} h_{22} - h_{13} h_{222}, \\
\label{dG_4_2} \nabla_{2} R_{3123} & = h_{232} h_{13} + h_{32} h_{213} - h_{233} h_{12} - h_{33} h_{212}, \\
\label{dG_4_3} \nabla_{3} R_{1223} & = h_{312} h_{23} + h_{12} h_{323} - h_{313} h_{22} - h_{13} h_{322}, \\
\label{dG_4_4} \nabla_{3} R_{3123} & = h_{332} h_{13} + h_{32} h_{313} - h_{333} h_{12} - h_{33} h_{312}. 
\end{align}
Since $R_{1223} = R_{1323} = 0$, $R_{1212} \neq 0$, $R_{2323} \neq 0$, the metric Lie algebra $(\g,\naiseki)$ 
admits a unique solution of the Gauss equation up to sign, and it is given by 
\begin{align*}
& h_{11}=\pm \sqrt{\frac{R_{1212} R_{1313}-R_{1213}^2}{R_{2323}}}, \
h_{22} = \frac{R_{1212}}{h_{11}} ,\
h_{33} = \frac{R_{1313}}{h_{11}} ,\\
& h_{23} = \frac{R_{1213}}{h_{11}} ,\
h_{12} = h_{13} = 0.
\end{align*}
Substituting $h_{12} = h_{13} = 0$ into (\ref{dG_4_1}) $\sim$ (\ref{dG_4_4}), we calculate 
two expressions
\begin{align*}
& (\ref{dG_4_1}) \times h_{33} + (\ref{dG_4_2}) \times h_{23}, \\ 
& (\ref{dG_4_3}) \times h_{23} + (\ref{dG_4_4}) \times h_{22}.
\end{align*}
Then we have
\begin{align*}
h_{213} 
& = -\frac{1}{h_{22}h_{33}-h_{23}^2}\left\{ h_{33} (\nabla_{2} R_{1223}) + h_{23} (\nabla_{2} R_{3123}) \right\} \\
& = -\frac{1}{R_{2323}}\left\{ h_{33} (\nabla_{2} R_{1223}) + h_{23} (\nabla_{2} R_{3123}) \right\}, \\
h_{312} & = -\frac{1}{h_{22}h_{33}-h_{23}^2}\left\{ h_{22} (\nabla_{3} R_{3123}) + h_{23} (\nabla_{3} R_{1223}) \right\} \\
& = -\frac{1}{R_{2323}}\left\{ h_{22} (\nabla_{3} R_{3123}) + h_{23} (\nabla_{3} R_{1223}) \right\}.
\end{align*}
(Here, we used the property $h_{23} = h_{32}$.)

On the other hand, by Lemma~\ref{R1} we have 
\begin{align*}
& \nabla_{2} R_{1223} = (b+c)(R_{1212}-R_{2323}) - aR_{1213},  \\
& \nabla_{2} R_{3123} = a(R_{1313}-R_{2323}) - (b+c)R_{1213},  \\
& \nabla_{3} R_{1223} = d(R_{1212}-R_{2323}) - (b+c)R_{1213}, \\
& \nabla_{3} R_{3123} = (b+c)(R_{1313}-R_{2323}) - dR_{1213}. 
\end{align*}
This yields that 
\begin{align*}
0&=h_{213}-h_{312}=\frac{1}{h_{11}} \left\{ 
(b+c)(R_{1313}-R_{1212})+(a-d)R_{1213}
\right\}.
\end{align*}
We thus obtain 
\begin{align*}
(b+c)(R_{1313}-R_{1212})+(a-d)R_{1213}=0.
\end{align*}
\end{proof}

\begin{prop}
Let $(\mathfrak{r}_3, \naiseki)$ be a solvable metric Lie algebra, admitting a solution of 
the Gauss equation in codimension $1$.
Then for any solution of the Gauss equation, $(\mathfrak{r}_3, \naiseki)$ does not have a 
solution of the derived Gauss equation. 
\end{prop}

\begin{proof}
Take any inner product $\naiseki$ on $\mathfrak{r}_3$.
Recall that, there exist $\lambda$ and a basis $\{ x_1,x_2,x_3 \}$, whose non-zero bracket relations are given by 
\begin{align*}
[x_1,x_2]=x_2+2 \lambda x_3, \ [x_1, x_3]=x_3.
\end{align*}
Since $(\mathfrak{r}_3,\naiseki)$ has a solution of the Gauss equation, we have 
$\frac{1}{\sqrt{3}} < \lambda <1$ by Proposition~\ref{G1}.
We apply Lemma~\ref{dGeq} for the case $a=d=1$, $b= \lambda$, $c=0$. 
Then it follows 
\begin{align*}
(b+c)(R_{1313}-R_{1212}) +(a-d)R_{1213}&= \lambda (R_{1313}-R_{1212}) \\
&=\lambda \{ -1+\lambda^2 - (-1-3\lambda^2)\} \\
&= 4 \lambda^3 \neq 0,
\end{align*}
which completes the proof.
\end{proof}

By applying Lemma~\ref{dGeq} and by similar arguments as above, one can show the following propositions. 

\begin{prop}
Let $\alpha \in [-1, 0) \cup (0,1)$ and let $(\mathfrak{r}_{3,\alpha}, \naiseki)$ be a solvable metric 
Lie algebra, admitting a solution of the Gauss equation in codimension $1$.
Then for any solution of the Gauss equation, $(\mathfrak{r}_{3,\alpha}, \naiseki)$ does not have a 
solution of the derived Gauss equation. 
\end{prop}

\begin{proof}
Take any inner product $\naiseki$ on $\mathfrak{r}_{3, \alpha}$.
Recall that, there exist $\lambda$, and a basis $\{ x_1,x_2,x_3 \}$ whose non-zero bracket relations are given by 
\begin{align*}
[x_1,x_2]=x_2+2 \lambda (\alpha-1) x_3, \ [x_1, x_3]= \alpha x_3. 
\end{align*}
By Proposition~\ref{G2}, we only consider the following cases:
\begin{align*}
& |\lambda| < \frac{\sqrt{\sqrt{(1+\alpha^2)^2+12 \alpha^2}-(1+\alpha^2)}}{\sqrt{6}(1-\alpha)} 
& & \mbox{when} \quad \alpha \in [-1, 0), \\
& \frac{\sqrt{\sqrt{(1+\alpha^2)^2+12 \alpha^2}-(1+\alpha^2)}}{\sqrt{6}(1-\alpha)}<|\lambda| 
< \frac{\sqrt{\alpha}}{1-\alpha} & & \mbox{when} \quad  \alpha \in (0, 1).
\end{align*} 

\noindent
We apply Lemma~\ref{dGeq} for $a=1$, $b=\lambda(\alpha-1)$, $c=0$ and $d=\alpha$.
In case $\lambda \neq 0$, we have 
\begin{align*}
(b+c)(R_{1313}-R_{1212}) +(a-d)R_{1213} &=\lambda (\alpha-1)(R_{1313}-R_{1212})+(1-\alpha)R_{1213} \\
&=\lambda(\alpha-1)^3(1+4\lambda^2) \neq 0,
\end{align*} 
which shows that $(\g, \naiseki)$ dose not admit a solution of the derived Gauss equation.

We next consider the case $\lambda =0$.
In this case we have $\alpha \in [-1, 0)$, and it is easy to check that $R_{1212}=-1$, 
$R_{1313}=-\alpha^2$, $R_{2323}=-\alpha$, $R_{1213}=R_{1223}=R_{1323}=0$.
Hence we have
$$
S = \frac{R_{1212}R_{1313}-R_{1213}^2}{R_{2323}} = -\alpha > 0,
$$
and by Lemma~\ref{Lem G}
we have $h_{11}, h_{22}, h_{33} \neq 0$, and $h_{12}=h_{13}=h_{23}=0$. 
As for the derivatives of $R$ we have $\nabla_p R_{ijkl} = 0$ except for two cases 
$\nabla_3 R_{1223} = \nabla_2 R_{1323} = \alpha(\alpha-1) \neq 0$. 
From the derived Gauss equation (\ref{9}) we thus obtain that
\begin{align*}
0=\nabla_1 R_{1212}=h_{111}h_{22}+h_{11}h_{122},\\
0=\nabla_1 R_{1313}=h_{111}h_{33}+h_{11}h_{133}.
\end{align*}
It follows from these equalities that 
\begin{align*}
h_{122}=- \frac{h_{111}h_{22}}{h_{11}}, \qquad 
h_{133}=- \frac{h_{111}h_{33}}{h_{11}}.
\end{align*}
By substituting these equalities into $0=\nabla_1 R_{2323}= h_{122}h_{33}+h_{22}h_{133}$, 
we obtain $h_{111}=0$, and this gives that $h_{122}=0$.
On the other hands, we see that
\begin{align*}
0 \neq \alpha(\alpha-1) & = \nabla_2 R_{1323} = h_{212}h_{33}+h_{12}h_{233}-h_{213}h_{32}-h_{13}h_{232} \\
& = h_{122}h_{33} = 0,
\end{align*}
which is a contradiction. We thus complete the proof.
\end{proof}

\begin{prop}
Let $\alpha >0$ and let $(\mathfrak{r}^{\prime}_{3,\alpha}, \naiseki)$ be a solvable metric 
Lie algebra, admitting a solution of the Gauss equation in codimension $1$.
Then for any solution of the Gauss equation, $(\mathfrak{r}^{\prime}_{3,\alpha}, \naiseki)$ does not 
have a solution of the derived Gauss equation. 
\end{prop}

\begin{proof}
Take any inner product $\naiseki$ on $\mathfrak{r}^{\prime}_{3, \alpha}$.
Recall that there exist $\lambda \geq 1$ and a basis $\{ x_1,x_2,x_3 \}$, whose non-zero bracket relations are given by 
\begin{align*}
[x_1,x_2]=\alpha x_2- \lambda x_3, \ [x_1, x_3]= \frac{1}{\lambda} x_{2}+ \alpha x_3.
\end{align*}
Then, by Proposition~\ref{G3} and the property $\alpha > 0$, we have 
\begin{align*}
1 < \sqrt{\frac{1+2\alpha^2 +2 \sqrt{1+\alpha^2+\alpha^4}}{3}} < \lambda. 
\end{align*}
We apply Lemma~\ref{dGeq} for the case $a= \alpha$, $b=-\lambda/2$, $c=1/(2\lambda)$, $d=\alpha$. 
Then by a direct calculation, we have 
\begin{align*}
(b+c)(R_{1313}-R_{1212})+(a-d)R_{1213}&=\left( -\frac{\lambda}{2}+\frac{1}{2 \lambda} \right)(R_{1313}-R_{1212}) \\
&=-\frac{1}{2}\left( \lambda-\frac{1}{\lambda} \right)^2 \left( \lambda+\frac{1}{\lambda} \right) \neq 0.
\end{align*}
We complete the proof.
\end{proof}

\subsection{Simple cases}
Throughout this subsection, we consider three-dimensional simple metric Lie algebra, 
which is not isometric to $(\mathfrak{so}(3), k\naiseki_{(0,2)})$ for any $k >0$.

Let $(\g, \naiseki)$ be a three-dimensional simple metric Lie algebra.
Recall that, there exist $\lambda_2 >0$, $\lambda_3 \neq 0$ and an orthonormal basis 
$\{ x_1, x_2, x_3 \}$ of $\g$ with respect to $\naiseki$ such that the bracket relations are given by 
\begin{align*} 
[x_1, x_2] = \lambda_3 x_3, \quad [x_2, x_3] = x_1, \quad [x_3, x_1]=\lambda_2 x_2. 
\end{align*}
Here we put 
$$
\lambda_2:=\frac{u+v}{2}, \qquad \lambda_3:=\frac{-u+v}{2}
$$
as in subsection 4.3. 
Remind that $u \neq \pm v$.

\begin{lem}\label{dGeqS}
Let $(\g, \naiseki)$ be a three-dimensional simple metric Lie algebra, which is not isometric 
to $(\mathfrak{so}(3), k\naiseki_{(0,2)})$ for any $k>0$. 
Assume that $(\g, \naiseki)$ has a solution of the Gauss equation in codimension $1$. 
If $(\g, \naiseki)$ has a solution of the derived Gauss equation, 
then the following equalities hold:
\begin{align*}
R_{2323} (\nabla_{1}R_{1213})&=-R_{1313} (\nabla_{2}R_{1223})=R_{1212} (\nabla_{3}R_{2313}).
\end{align*}
\end{lem}
\begin{proof}
By the assumption that $(\g, \naiseki)$ has a solution of the derived Gauss equation, 
the following equalities hold:
\begin{align*}
\nabla_1 R_{1213}& = h_{111} h_{23}+h_{11} h_{123} - h_{121} h_{13} - h_{21} h_{113},\\
\nabla_2 R_{1223}& = h_{212} h_{23}+h_{12} h_{223} - h_{222} h_{13} - h_{22} h_{213},\\
\nabla_3 R_{2313}& = h_{321} h_{33}+h_{21} h_{333} - h_{331} h_{23} - h_{31} h_{323},  
\end{align*}
where $h_{ij}$ is the solution of the Gauss equation given by
\begin{align*}
&h_{11}=\pm \sqrt{\frac{R_{1212} R_{1313}}{R_{2323}}}, \
h_{22} = \frac{R_{1212}}{h_{11}} ,\
h_{33} = \frac{R_{1313}}{h_{11}} ,\
h_{23} = h_{12} = h_{13} = 0.
\end{align*}
Note that $h_{ij}$ is uniquely determined up to sign, and satisfies $h_{11} \neq 0$, $h_{22} \neq 0$, 
$h_{33} \neq 0$.
These yield that 
\begin{align*}
h_{123}=\frac{\nabla_{1} R_{1213}}{h_{11}},\
h_{213}=-\frac{h_{11} (\nabla_{2} R_{1223})}{R_{1212}},\
h_{321}=\frac{h_{11} (\nabla_{3} R_{2313})}{R_{1313}}.
\end{align*}
Since $h_{ijk}$ is a symmetric $3$-tensor, we have $h_{123}=h_{213}=h_{321}$.
By these conditions, we can immediately show that 
\begin{align*}
R_{2323} (\nabla_{1}R_{1213})&=-R_{1313} (\nabla_{2}R_{1223})=R_{1212} (\nabla_{3}R_{2313}).
\end{align*}
\end{proof}

\begin{prop}
Let $(\g, \naiseki)$ be a three-dimensional simple metric Lie algebra, which is not 
isometric to $(\mathfrak{so}(3), k\naiseki_{(0,2)})$ for any $k >0$.
Assume that $(\g, \naiseki)$ has a solution of the Gauss equation in codimension $1$.
Then $(\g, \naiseki)$ does not have a solution of the derived Gauss equation for any solution 
of the Gauss equation.
\end{prop}

\begin{proof}
We use the same notations as above.
From the proof of Proposition~\ref{G_simple} we have 
\begin{align*}
& R_{1212} = \frac{1}{4}\{ 1-u^2+2(v-1)u \}, \\
& R_{1313} = \frac{1}{4}\{ 1-u^2-2(v-1)u \}, \\ 
& R_{2323} = -\frac{1}{4}\{ 1-u^2-2(v-1) \},
\end{align*}

\noindent
and by Lemma~\ref{R2} we have 
\begin{align*}
\nabla_{1} R_{1213} &=
\frac{1}{2}u(v-1)^2, \\
\nabla_{2} R_{1223} &=
\frac{1}{4}(u-1)^2(u-v+2), \\
\nabla_{3} R_{2313} &=
-\frac{1}{4}(u+1)^2(u+v-2).  \\
\end{align*}
Assume that the equalities 
\begin{align}\label{C}
R_{2323} (\nabla_{1}R_{1213})&=-R_{1313} (\nabla_{2}R_{1223})=R_{1212} (\nabla_{3}R_{2313})
\end{align}
hold.
Note that, in the case of $u=0$, it is easy to see that the conditions~(\ref{C}) hold if and only if $v=2$, 
and we have excluded this case in advance.
Hereinafter we suppose that $u \neq 0$.
Then, by the conditions~(\ref{C}), we obtain that 
\begin{align*}
0&=\{ R_{2323} (\nabla_{1}R_{1213})-R_{1212} (\nabla_{3}R_{2313}) \} \\
& \qquad \qquad \qquad -\frac{u+v}{2} \left\{ -R_{1313} (\nabla_{2}R_{1223})-R_{1212} (\nabla_{3}R_{2313}) \right\}\\
&=-\frac{1}{16}(u - 1)(v - 1)l_1,\\ 
\rule{0cm}{0.7cm} 0&= \{ R_{2323} (\nabla_{1}R_{1213})-R_{1212} (\nabla_{3}R_{2313}) \} \\
& \qquad \qquad \qquad +\frac{u+v}{2} \left\{ -R_{1313} (\nabla_{2}R_{1223})-R_{1212} (\nabla_{3}R_{2313})\right\}\\
&=\frac{1}{16}(u + 1)(u+v-2)l_2,
\end{align*}
where
\begin{align*}
l_1:=& u^4+(v-1)u^3+(v-3)u^2+(4 v^2-9 v+5)u-v+2,\\
l_2:=& (v-3)u^3+(3 v-5)u^2+(4 v^2-5v+3)u+v+1.
\end{align*} 
By conditions~(\ref{C}) we can easily show that $u= \pm 1$ if and only if $v=1$.
This contradicts the assumption $u \neq \pm v$, and hence we have $u \neq \pm 1$, $v \neq 1$.
In addition we can show $u+v \neq 2$ by using the assumption $u \neq 0$. 
Therefore we have $l_1= l_2 =0$.

Then it follows that 
\begin{align*}
0=l_1-l_2=(u+1)^2(u^2-2v+1).
\end{align*} 
In case $u^2-2v+1=0$, direct calculations show that 
\begin{align*}
& 0 = l_1 = l_2 = \frac{3}{2}(u-1)^2(u+1)^3, 
\end{align*}
which contradicts $u \neq \pm 1$.
Consequently the conditions~(\ref{C}) do not hold when $u \neq 0$.
\end{proof}

\end{document}